\documentclass[12pt]{amsart}
\usepackage[usenames]{color}
\usepackage{amsmath,verbatim,graphicx,epstopdf,enumerate}
\usepackage{tabularx}
\usepackage{booktabs}
\usepackage{amsmath,amsfonts,amsthm,mathrsfs,amssymb,cite}
\usepackage[usenames]{color}
\usepackage[notref,notcite]{}
\usepackage[pagewise]{lineno}
\newtheorem{theorem}{Theorem}[section]
\newtheorem{lemma}[theorem]{Lemma}

\newtheorem{corollary}[theorem]{Corollary}

\newtheorem{remark}[theorem]{Remark}

\numberwithin{equation}{section}

\usepackage{hyperref}

\newcommand{\PD}{\partial}

\newcommand{\bp}{\begin{prob}}
	\newcommand{\ep}{\end{prob}}
\newcommand{\bpr}{\begin{proof}}
	\newcommand{\epr}{\end{proof}}

\newcommand{\lb}{\left(}

\newcommand{\rb}{\right)}

\newcommand{\Cb}{\mathbb{C}}

\newcommand{\Rb}{\mathbb{R}}
\newcommand{\Sb}{\mathbb{S}}

\newcommand{\D}{\mathrm{d}}

\usepackage{amsmath}
\theoremstyle{definition}

\newcommand{\bel}[1]{\begin{equation}\label{#1}}
	\newcommand{\ee}{\end{equation}}

\newcommand{\be}{\begin{eqnarray}}
\newcommand{\ben}{\begin{eqnarray*}}
\newcommand{\en}{\end{eqnarray}} 
\newcommand{\enn}{\end{eqnarray*}}
\title[Inverse electromagnetic scattering]{Inverse time-harmonic electromagnetic scattering from coated polyhedral scatterers with a single far-field pattern}
\author[Hu, Vashisth and Yang]{Guang-Hui Hu$^{\dagger}$, Manmohan Vashisth$^{\ddagger}$ and Jiaqing Yang$^{*}$ }
\address{$^{\dagger}$School of Mathematical Sciences and LPMC, Nankai University, Tianjin 300071, China
	\newline
	\indent E-mail: {\tt ghhu@nankai.edu.cn}}
\address{$^{\ddagger}$Department of Mathematics, Indian Institute of Technology, Jammu  181221, India
	\newline 
	\indent E-mail: {\tt  manmohan.vashisth@iitjammu.ac.in}}
\address{$^{*}$School of Mathematics and Statistics,
Xi'an Jiaotong University, Xi'an 710049,  P. R. China.
	\newline
	\indent E-mail: {\tt  jiaq.yang@mail.xjtu.edu.cn}}

\begin{document}
	\maketitle
	
	\begin{abstract}    
	It is proved that a convex polyhedral scatterer of impedance type can be uniquely determined by the electric far-field pattern of a non-vanishing incident field. The incoming wave is allowed to { be} an electromagnetic plane wave, a vector Herglotz wave function or a point source wave incited by some magnetic dipole.
%single incident plane wave with fixed direction, polarization and wavenumber.
Our proof relies on the reflection principle for Maxwell's equations with the impedance (or Leontovich) boundary condition enforcing on a hyper-plane. We prove that it is impossible to analytically extend the total field across any vertex of the scatterer. This leads to a data-driven inversion scheme for imaging an arbitrary convex polyhedron.
	\end{abstract}\vspace{2mm}
	
	\vspace{2mm} 
	
	   
	\section{Introduction and Main Result}
The propagation of time-harmonic electromagnetic waves in a homogeneous isotropic medium in $\Rb^{3}$ is modelled by the Maxwell's equations
\begin{align}\label{Maxwell equation in R3}
\nabla\times E(x) -ik H(x)=0, \ \nabla \times H(x)+ik E(x)=0\qquad \mbox{for}\ x\in\Rb^{3},
\end{align}
where  $E$ and $H$  represent the electric and magnetic fields respectively and $k>0$ is known as the wave number. Let $E^{in}$ and $H^{in}$ satisfying Equation  \eqref{Maxwell equation in R3} denote  the incident electric and magnetic fields respectively. Consider the scattering of given incoming waves $E^{in}$ and $H^{in}$ from a convex polyhedral scatterer $D\subset \Rb^3$ coated by a thin dielectric layer, which can be modelled by the impedance (or Leontovich) boundary value problem of the Maxwell's equations \eqref{Maxwell equation in R3}
in $\Rb^3\backslash\overline{D}$.
Then the total fields $E=E^{in}+E^{sc}$,  $H=H^{in}+H^{sc}$, where $E^{sc}$ and $H^{sc}$ denote the scattered fields, are governed by the following set of Equations \eqref{Maxwell equation in exterior}-\eqref{Boundary condition scattering}:
\begin{align}
&\label{Maxwell equation in exterior} \nabla \times E-ik H=0,\ \ \ \ \nabla\times H+ik E=0\ \;\mbox{in}\;\Rb^{3}\backslash{\overline{D}},\\
&\label{Total field} E= E^{in}+E^{sc},\ H=H^{in}+H^{sc}\ \mbox{in}\ \Rb^{3}\backslash{\overline{D}},\\
&\label{Radiation condition} \lim_{\lvert x\rvert \rightarrow \infty}\lb H^{sc}\times x-\lvert x\rvert E^{sc}\rb =0,\\
&\label{Boundary condition scattering} \nu\times\lb \nabla\times E\rb +i\lambda\nu\times\lb\nu\times E\rb=0\ \mbox{on} \ \PD D,
\end{align}
where $\nu$ denotes the outward unit normal to $\PD D$ and the impedance coefficient $\lambda>0$ is supposed to be a constant.
Equation \eqref{Radiation condition} is known as the Silver-M\"uller radiation condition and is uniform in all directions $\widehat{x}:=x/|x|$. For the  existence and uniqueness of the solution $(E,H)$ to the  forward system \eqref{Maxwell equation in exterior}-\eqref{Boundary condition scattering}, we refer to  \cite{Colton_Kress_Book} when $\PD D$ is $C^{2}$-smooth and to \cite{CCM,CHP} when $\PD D$ is Lipschitz with connected exterior. Moreover, the Silver-M\"uller radiation condition \eqref{Radiation condition}
ensures that scattered fields $E^{sc}$ and $H^{sc}$ satisfy the following asymptotic behaviour (see \cite{Colton_Kress_Book})
\begin{align}
&\label{Asymptotic expansion electric}E^{sc}(x)=\frac{e^{ik\lvert x\rvert}}{\lvert x\rvert}\lb E^{\infty}\lb\widehat{x}\rb +O\lb \frac{1}{\lvert x\rvert}\rb\rb,\ \mbox{as $\lvert x\rvert \rightarrow \infty$},\\
&\label{Asymptotic expansion magnetic} H^{sc}(x)=\frac{e^{ik\lvert x\rvert}}{\lvert x\rvert}\lb H^{\infty}\lb \widehat{x}\rb +O\lb \frac{1}{\lvert x\rvert}\rb\rb,\ \mbox{as $\lvert x\rvert \rightarrow \infty$}
\end{align}
where the vector fields $E^{\infty}$ and $H^{\infty}$ defined on the unit sphere $\Sb^{2}$, are called the electric and magnetic far-field patterns of the scattered waves $E^{sc}$ and $H^{sc}$,  respectively. It is well known that $E^{\infty}$ and $H^{\infty}$ are analytic functions with respect to the observation direction $\widehat{x}\in \Sb^2$
and satisfy the following relations
\begin{align}\label{Properties of far-field pattern}
H^{\infty}=\nu\times E^{\infty},\;\; \nu\cdot E^{\infty}=\nu\cdot H^{\infty}=0,
\end{align}
where $\nu$ denotes the unit normal vector to the unit sphere $\Sb^{2}$.

Given the incoming wave $(E^{in},H^{in})$   and the scatterer $D\subset\Rb^{3}$,
the direct problem arising from electromagnetic scattering is to find the scattered fields $(E^{sc}, H^{sc})$ and  their  far-field patterns. The inverse problem to be considered in this paper consists of determining the location and shape of $D$  from knowledge of the far-field patterns $(E^\infty, H^\infty)$.
We assume that the incident fields $E^{in}$ and $H^{in}$ are non-vanishing vector fields which are solutions to the Maxwell's equations \eqref{Maxwell equation in R3} in a neighboring area of the obstacle $D$.
For instance one can take the incident fields $E^{in}$ and $H^{in}$ to be one among the following:
\begin{enumerate}
	\item \textbf{Plane waves:}
	\begin{align}\label{Plane wave field}
	E^{in}(x,d,p)=pe^{ik x\cdot d},\qquad  H^{in}(x,d,p)=(d\times p)e^{ik x\cdot d}
	\end{align}
	where $d\in\Sb^{2}$ is known as the incident direction and $p\in \Sb^{2}$ with $p\perp d$ is known as the polarization direction.
	\item \textbf{Point source waves:}
	\begin{align}\label{Point source waves}
	E^{in}=\nabla\times \lb \Phi(x,y)\vec{a}\rb, \ \ H^{in}(x)=\frac{1}{ik}\nabla\times E^{in}(x),\ \ x\neq y,
	\end{align}
	where $\vec{a}$ is a constant vector and $\Phi(x,y):= \frac{1}{4\pi}\frac{e^{ik\lvert x-y\rvert}}{\lvert x-y\rvert}$,\ $x\neq y$ is the fundamental solution to $\Delta +k^{2}$ in $\Rb^{3}$. $\lb E^{in},H^{in}\rb$ given by (\ref{Point source waves}) represent the electromagnetic field generated by a magnetic dipole located at $y$ and they solve the Maxwell's equations \eqref{Maxwell equation in R3} for $x\neq y$.
	\item \textbf{Electromagnetic Herglotz pairs:}
	\begin{align}\label{Herglotz wave fields}
	E^{in}(x)=\int\limits\limits_{\Sb^{2}}e^{ikx\cdot d}a(d)ds(d), \ \ H^{in}(x)=\frac{1}{ik}\nabla\times E^{in}(x), \ x\in\Rb^{3},
	\end{align}
	where the square integrable tangential field  $a\in L_t^2(\Sb^{2})^3$ is known as the   Herglotz kernel of $E^{in}$ and $H^{in}$.
\end{enumerate}
%For a given function $F$, we use the notation  $F(x;D,p,k,d,\lambda)$ to represent the   dependence of  $F$ on the scatterer $D$, the polarization $p$, the wave number $k$, incident direction $d$ and the impedance coefficient $\lambda$.
%Let the incident field $\lb E^{in},H^{in}\rb$ be one among the given in \eqref{Plane wave field}, \eqref{Point source waves} and  \eqref{Herglotz wave fields}.
The present article is concerned with a uniqueness result of determining  the convex polyhedral scatterer $D$ appearing in the system of Equations \eqref{Maxwell equation in exterior}-\eqref{Boundary condition scattering} from the knowledge of a single electric far-field pattern $E^{\infty}(\widehat{x})$ measured over all observation directions $\widehat{x}\in\Sb^2$. We mention that  in  case of plane wave incidence, the incident direction $d\in \Sb^{2}$, the polarization direction $p\in \Sb^{2}$ and the wave number $k>0$  are all  fixed.

We state our main result as follows:
\begin{theorem}\label{Main theorem scattering}
	Let $D_{1}$ and $D_{2}$ be two convex polyhedral scatterers of impedance type. Given an incident field $(E^{in},H^{in})$ as mentioned above,  we denote by $ E_{j}^{\infty}$ ($j=1,2$) the electric far-field patterns of the scattering problem  \eqref{Maxwell equation in exterior}-\eqref{Boundary condition scattering} when $D=D_{j}$. Then the relation
	\begin{align}\label{Equality of far-field patterns}
	E^{\infty}_{1}(\widehat{x}) = E^{\infty}_{2}(\widehat{x})\quad \mbox{for all} \ \widehat{x}\in\Sb^{2}
	\end{align}
	implies that  $D_{1}=D_{2}$.
\end{theorem}
%\begin{theorem}\label{Main theorem scattering plane wave}
%	Let $D_{1}$ and $D_{2}$ be two convex polyhedral scatterers of impedance type. For fixed incident plane waves $(E^{in},H^{in})$, we  denote by $ E_{j}^{\infty}(\widehat{x};k,p,d)$ ($j=1,2$) the electric far-field patterns of the scattering problem  \eqref{Maxwell equation in exterior}-\eqref{Boundary condition scattering} when $D=D_{j}$. Then the relation
%	\begin{align}\label{Equality of far-field patterns plane wave}
%		E^{\infty}_{1}(\widehat{x};k,p,d) = E^{\infty}_{2}(\widehat{x};k,p,d)\quad \mbox{for all} \ \widehat{x}\in\Sb^{2}
%	\end{align}
%	implies that  $D_{1}=D_{2}$.
%\end{theorem}
It is widely open how to uniquely determine the shape of a general impenetrable/penetrable scatterer using a single far-field pattern. {As in the acoustic case \cite{Alessandrini_Rondi_acoustic,BPS, Cheng_Yamamoto, LiuC}},  quite limited progress has also been made in inverse time-harmonic electromagnetic scattering. To the best of our knowledge,  global uniqueness with a single measurement data is proved only for perfectly conducting obstacles with restrictive geometric shapes such like balls \cite{Kress2001} and convex polyhedrons \cite[Chapter 7.1]{Colton_Kress_Book}.
Without the convexity condition, it was shown in \cite{Liu_Yamamoto_Zou_Reflection principle_sound-soft_Maxwell} that a \emph{general} perfect polyhedral conductor (the closure of which may contain screens) can be uniquely determined
by the far-field pattern for plane wave incidence with one direction and two polarizations.  We shall prove Theorem \ref{Main theorem scattering} by using the reflection principle for Maxwell's equations with the impedance boundary condition enforcing on a hyper-plane. It seems that such a reflection principle has not been studied in prior works, although the corresponding principle under the perfectly conducting boundary condition is well known in optics (see e.g.\cite{Liu_Yamamoto_Zou_Reflection principle_sound-soft_Maxwell}).
Theorem \ref{Main theorem scattering} carries over to perfectly conducting polyhedrons with a single far-field pattern (see Corollary \ref{PEC}), and thus improves the acoustic uniqueness result for impedance scatterers \cite{Cheng2003} where two incident directions were used.
It is also worth mentioning other works in the literature related to  reflection principles for the Helmholtz and Navier equations together with their applications to uniqueness in inverse acoustic and elastic scattering \cite{Alessandrini_Rondi_acoustic,Cheng_Yamamoto,Colton1977, Elshner_Yamamoto_polygonal,EY2010,Elschner_Hu_Elastic,  Liu_Zou_}.
%The unique determination of non-convex polygons and polyhedrons of impedance type with a single far-field pattern still seems open.
We believe that the reflection principle, as a special case of unique continuation, provides a powerful tool for gaining new insights into inverse scattering problems.
More remarks concerning our uniqueness proof will be concluded in Section \ref{sec:remarks}.

In the second part of this paper, we shall propose a novel non-iterative scheme for imaging an arbitrary convex-polyhedron from a single electric far-field pattern.
The Linear Sampling Method in inverse electromagnetic scattering \cite{Cakoni,CCM,CHM,CHP} was earlier studied with infinite number of plane waves at a fixed energy.
We are mostly motivated by the uniqueness proof of Theorem \ref{Main theorem scattering} (see also Corollary \ref{singular}) and the one-wave factorization method in inverse elastic and acoustic scattering \cite{Elschner_Hu_Elastic, HuLi2020}. The proposed scheme is essentially a domain-defined sampling approach, requiring no forward solvers. Promising features of our imaging scheme are summarised as follows.

(i) It requires lower computational cost and only a single measurement data. The proposed domain-defined indicator function involves only inner product calculations. Since the number of sampling variables is comparable with the original Linear Sampling Method and Factorization Method (\cite{Cakoni, Colton_Kress_Book, Kirsch08}),
the computational cost is not heavier than the aforementioned pointwise-defined sampling methods.
(ii) It can be interpreted as a data-driven approach, because it relies on measurement data corresponding to a priori given scatterers (which are also called test domains in the literature or samples in the terminology of learning theory and data science). There is a variety of choices on the shape and physical properties of these samples, giving arise to quite \textquotedblleft rich\textquotedblright a priori sample data in addition to the measurement data of the unknown target. In this paper, we choose perfectly conducting balls (or impedance balls) as test domains, because the spectra of the resulting far-field operator admit explicit representations. However, these test domains can also be chosen as any other convex penetrable and impenetrable scatterers, provided the classical factorization scheme for imaging this test domain can be verified using all incident and polarization directions. We refer to \cite{Kirch2004} for the Factorization Method applied to inverse electromagnetic medium scattering problems.
(iii) It provides a necessary and sufficient criterion for imaging convex polyhedrons (see Theorem \ref{th4.2}) with a single incoming wave. We prove that the wave fields cannot be analytic around any vertex of $D$ (see Corollary \ref{singular}),  excluding the possibility of analytical extension across a vertex. Some other domain-defined sampling approaches such as the range test approach  \cite{KPS, KS} and the one-wave no-response test \cite{P2003,P2007}  usually pre-assume such extensions, leading to a sufficient condition for imaging general targets. Our approach is closest to the No Response Test for reconstructing perfectly conducting polyhedral scatterers with a few incident plane waves \cite{PS2010} and is comparable with the
one-wave enclosure method by Ikehata \cite{Ik2000, Ik1999} for capturing singular points of $\partial D$. {Detailed discussions on the issue of analytic continuation tests can be found in the monograph \cite[Chapter 15]{NP2013} (see also \cite{HNS}).} If $\partial D$ contains no singular points, only partial information of $D$ can be numerically recovered; see \cite{LS} where the linear sampling method with a single far-field pattern was tested.

We organize the article as follows.  In \S \ref{Uniqueness with plane wave} we prove Theorem \ref{Main theorem scattering} when the incident fields are given by \eqref{Plane wave field}. In \S \ref{Reflection principle} we state and prove the reflection principle for Maxwell's equations with the impedance  boundary condition on a hyper-plane in $\Rb^{3}$ (see Theorem \ref{Main theorem reflection}).  Using this reflection principle we prove in \S \ref{Main theorem} the main  uniqueness results for electromagnetic Herglotz waves and point source waves. The data-driven reconstruction scheme will be described in \S \ref{numeric}.
	   	\section{Uniqueness with a Single Plane Wave}\label{Uniqueness with plane wave}
	   In this section, we prove  Theorem \ref{Main theorem scattering}, when the incident fields $(E^{in},H^{in})$ are given by
	   \begin{align*}
	   E^{in}(x,d,p)=pe^{ik x\cdot d},\qquad  H^{in}(x,d,p)=(d\times p)e^{ik x\cdot d}
	   \end{align*}
	   where $d,p\in \Sb^{2}$ satisfying $p\perp d$ and $k>0$ are all fixed.
	   Now recall from Equation \eqref{Equality of far-field patterns} that  $E_{1}^{\infty}(\widehat{x})=E_{2}^{\infty}(\widehat{x})$ for all $\widehat{x}\in\Sb^{2}$. Using the Rellich's lemma (see \cite{Colton_Kress_Book}) we get
	   \[E_{1}=E_{2}\ \mbox{and}\ H_{1}=H_{2}\ \mbox{in}\ \Rb^{3}\setminus\lb \overline{D_{1}\cup D_{2}}\rb.\]
	   Assuming that $D_{1}\neq D_{2}$, we shall prove the uniqueness by deriving a contradiction. By the convexity of $D_1$ and $D_2$,  we may assume that  there exists a vertex $O$ of $\PD D_{1}$ and a neighborhood $V_{O}$ of $O$ such that $V_{O}\cap \overline{{D}_{2}}=\emptyset$. Next using the impedance boundary condition of $E_1$ on $\PD D_{1}$ and $E_{1}=E_{2}$ in $\Rb^{3}\setminus{\lb \overline{{D}_{1}\cup D_{2}}\rb}$, we have that $\nu\times \lb \nabla\times E_{2}\rb+i\lambda\nu\times\lb\nu\times E_{2}\rb=0$  on $V_{O}\cap \PD D_{1}$.
	   Since $D_{1}$ is a convex polyhedron,  there exists $m$ ($m\geq 3$) convex polygonal faces $\Lambda_{j}$ ($j=1,2\cdots m$) of $\partial D_{1}$ whose closure meet at  $O$ and which can be analytically extended to infinity in $\Rb^3\backslash\overline{D}_2$; for example, see Figure \ref{f1} on page 12, where $m=4$ (left) and $m=3$ (right). Denote by $\widetilde{\Pi}\supseteq \Lambda_j$ the maximum extension of
	   $\Lambda_j$ in $\Rb^3\backslash\overline{D}_2$. %Now because of convexity $D_{2}$ will be completely on one side of these polygonal faces $\Gamma_{j}$ thereofore
	   Using the real-analyticity of $E_{2}$ in $\Rb^3\backslash\overline{D}_2$ together with the fact that $\lambda>0$ is a constant, we conclude that $E_{2}$ satisfies the impedance boundary condition on $\widetilde{\Pi}_j$. Recalling \eqref{Radiation condition} and \eqref{Asymptotic expansion electric}, we have
	   $\lim\limits_{\lvert x\rvert\rightarrow \infty}\lvert  \nabla\times E^{sc}_{2}\rvert=0, \ \ \  \lim\limits_{\lvert x\rvert\rightarrow \infty}\lvert E^{sc}_{2}\rvert=0.$
	   Hence,%Using this together with $\nu\times \lb \nabla\times E_{2}\rb+i\lambda\nu\times\lb\nu\times E_{2}\rb=0$  on $\Pi_{j}$ for $ j\geq 3,$ we have
	   \begin{align}\label{Impedance BC for incident}
	   \nu\times \lb \nabla\times E^{in}\rb+i\lambda\nu\times\lb\nu\times E^{in}\rb=0 \mbox{ on} \  \widetilde{\Pi}_{j},\quad   j=1,2,\cdots m.
	   \end{align}
	   By Equation \eqref{Plane wave field}, it then follows that
	   \begin{align}\label{pd}
	   ik \nu\times\lb d\times p\rb +i\lambda \nu\times \lb \nu\times p\rb =0,\
	   \end{align}
	   holds for any outward unit normal $\nu$ to $\widetilde{\Pi}_{j}$. Without loss of generality, we suppose that $p=\textbf{e}_{1}$,  $d\times p=\textbf{e}_{2}$,  $\nu=c_{1}\textbf{e}_{1}+c_{2}\textbf{e}_{2}+c_{3}\textbf{e}_{3}$ with $c_{1}^{2}+c_{2}^{2}+c_{3}^{2}=1$, where $\textbf{e}_j\in \Sb^{2}$ ($j=1,2,3$) denotes the Cartesian coordinates in $\Rb^3$.
	   By \eqref{pd}, simple calculations show that
	   
	   \[	-\lb kc_{3}+\lambda(c_{2}^{2}+c_{3}^{2})\rb \textbf{e}_{1}+\lambda c_{1}c_{2}\textbf{e}_{2}+\lb kc_{1}+\lambda c_{1}c_{3}\rb \textbf{e}_{3}=0.\]
	   This gives us the equations of $c_j$, $$kc_{3}=-\lambda \lb c_{2}^{2}+c_{3}^{2}\rb,\quad \lambda c_{1}c_{2}=0,\quad kc_{1}=-\lambda c_{1}c_{3},$$
	   which have the following solutions for $\nu=(c_1,c_2,c_3)\in \Sb^2$:
	   \ben
	   \nu&=&(0,0,-1),\qquad\qquad\,\,\qquad\qquad\quad\mbox{if}\quad k=\lambda;\\
	   \nu&=&
	   \lb \pm\sqrt{1-(k/\lambda)^2}, 0,-k/\lambda\rb ,\qquad \mbox{if}\quad k<\lambda;\\
	   \nu&=&
	   \lb 0, \pm\sqrt{1-(\lambda/k)^2},-\lambda/k\rb ,\qquad \mbox{if}\quad k>\lambda.
	   \enn
	   Hence, the relation \eqref{pd} cannot hold for three linearly independent unit normal vectors $\nu$.
	   This contradiction implies that $D_{1}=D_{2}$.
	   \begin{remark}
	   	The above uniqueness proof  with a single plane wave cannot be applied  to the Helmholtz equation in two dimensions under the impedance boundary condition $\PD_{\nu}u+i\lambda u=0$. In 2D case, we deduce a corresponding relation $k\nu\cdot d+\lambda =0$, which holds for only two linearly independent unit normal vectors if $D_1\neq D_2$. However, this cannot lead to a contradiction when $k>\lambda$. It was proved in \cite{Cheng2003} that the far-field patterns of two incident directions uniquely determine a convex polygonal obstacle of impedance type.
	   \end{remark}
	   
	   Since the above proof relies on the form of electromagnetic plane waves therefore it is not applicable to the incident fields given by electromagnetic point source waves and vector Herglotz wave functions. For these kind of incident fields, we shall apply the reflection principle for Maxwell's equations to prove Theorem \ref{Main theorem scattering}; see Sections \ref{Reflection principle} and subsection \ref{Main theorem} below.
	\section{Reflection Principle for Maxwell's Equations}\label{Reflection principle}
Let $\Omega\subseteq \Rb^{3}$ be an open connected set which is symmetric with respect to a plane $\Pi$ in $\Rb^{3}$ and we define by  $\gamma:=\Omega\cap \Pi$.
Denote by $\Omega^{+}$ and $\Omega^{-}$  the two symmetric parts of $\Omega$  which are divided by $\Pi$ and by $R_{\Pi}$ the reflection operator about $\Pi$, that is, if $x\in\Omega^{\pm}$ then $R_{\Pi}x\in\Omega^{\mp}$ for $x=(x_1, x_2, x_3)\in \Omega$.
Throughout this article,    $\Omega$ will be assumed in such a way that any line segment with end points in $\Omega$ and intersected with $\Pi$ by the angle $\pi/2$ lies completely in $\Omega$. In other words, the projection of any line segment in $\Omega$ onto the hyperplane $\Pi$ is a subset of $\gamma$.
This geometrical condition was also used in \cite{Diaz_Ludford-Reflection principle} where the reflection principle for the Helmholtz equation with the impedance boundary condition was verified in $\Rb^n$ ($n\geq 2$).
Now consider the time-harmonic  Maxwell's equations with the impedance boundary condition by
\begin{align}
\label{Main equation}&\nabla\times E-ikH=0,\ \ \nabla\times H+ik E=0,\ && \mbox{in}\quad\Omega^{+},\\
\label{Boundary condition}&\nu\times\lb\nabla\times E\rb+i\lambda\nu\times\lb\nu\times E\rb=0,\ && \mbox{on}\quad \gamma\subset \Pi.
\end{align}
It is well known from  (Theorem $6.4$ in \cite{Colton_Kress_Book}) that a solution $(E,H)$ of  Equation \eqref{Main equation} satisfies the vectorial Helmholtz equations with the divergence-free condition:
\begin{align}\label{Divergence free and Helmholtz equation}
\Delta E+k^{2}E=0,\ \Delta H+k^{2}H=0,\quad \nabla\cdot E=\nabla\cdot H=0.
\end{align}
	Since Equations \eqref{Divergence free and Helmholtz equation} and  \eqref{Main equation} are rotational invariant,  without loss of generality we can assume that the plane $\Pi$ mentioned above coincides with the $ox_1x_2$-plane, i.e., $\Pi=\{(x_{1},x_{2},x_{3})\in\Rb^{3}:\ x_{3}=0\}$. Consequently, we have $\nu=\textbf{e}_3:=(0,0,1)^T$ and $R_\Pi x=(x_1, x_2, -x_3)$.
In this  section, we study the reflection principle for solutions to the Maxwell's equation satisfying the impedance  boundary condition on $\gamma$. Our aim is to
extend the solution $(E,H)$ of Equation \eqref{Main equation} from $\Omega^{+}$ to $\Omega^{-}$ by an analytical formula. The reflection principle is stated as follows.
	\begin{theorem}\label{Main theorem reflection}
	Let $\Pi:= \{\lb x',x_{3}\rb\in\Rb^{3}: x_{3}=0\}$ with $x':=(x_1, x_2)$,  $\gamma \subset \Omega\cap \Pi$ and $\Omega^{\pm}:=\{(x_{1},x_{2},x_{3})\in\Omega:\ \pm x_{3}>0\}$. Assume that  $(E,H)$ satisfies Equation \eqref{Main equation} with the boundary condition \eqref{Boundary condition}. Then $(E,H)$ can be analytically extended to $\Omega^{-}$ as a solution to \eqref{Main equation}. Moreover, the extended electric field $\widetilde{E}:=\lb\widetilde{E}_{1},\widetilde{E}_{2},\widetilde{E}_{3}\rb^{T}$ is given explicitly by
	\begin{align}\label{Extension formula}
	\begin{aligned}
	\widetilde{E}(x',x_{3})=
	\begin{cases}
	E(x),\ \ \mbox{if}\ x\in \Omega^{+}\cup\gamma\\
	\mathcal{D}E(x',-x_{3}),\ \ \mbox{if}\ x\in \Omega^{-}
	\end{cases}
	\end{aligned}
	\end{align}
	where the operator $\mathcal{D}E:=\left( (\mathcal{D}E)_{1},(\mathcal{D}E)_{2},(\mathcal{D}E)_{3}\right)^{T}$ is defined by
	\be\label{E3 tilde extension Theorem}
	\ \ \ &&(\mathcal{D}{E})_{3}(x)
	:=E_{3}(x',x_{3})-\frac{2k^{2}}{i\lambda}\int\limits\limits_{0}^{x_{3}}e^{\frac{k^{2}}{i\lambda}(s-x_3)}E_{3}(x',s)\D s\ \ \mbox{for} \  x\in\Omega^{+},\\
	\label{Extension formula for Ej Theorem}
	&&(\mathcal{D}{E})_{j}(x',x_{3}):=
	E_{j}(x',x_{3})+2i\lambda \int\limits\limits_{0}^{x_{3}}e^{-i\lambda (s-x_3)}E_{j}(x',s)\D s \\ \nonumber
	&& +\frac{2\lambda^{2}}{k^{2}-\lambda^{2}}\int\limits\limits_{0}^{x_{3}}e^{-i\lambda (s-x_3)}\PD_{j}E_{3}(x',s)\D s
	-\frac{2k^{2}}{k^{2}-\lambda^{2}}\int\limits\limits_{0}^{x_{3}}e^{\frac{k^{2}}{i\lambda}(s-x_3)}\PD_{j}E_{3}(x',s)\D s
	\en
	for $j=1,2$ and  $x\in\Omega^{+}$.
	
	%	$\widetilde{E}:=E$ in $\Omega^{+}\cup\gamma$ and
	%\be\label{E3 tilde extension Theorem}
	%&&\widetilde{E}_{3}(x)
	%:=E_{3}(x',-x_{3})-\frac{2k^{2}}{i\lambda}\int\limits\limits_{0}^{-x_{3}}e^{\frac{k^{2}}{i\lambda}(s+x_3)}E_{3}(x',s)\D s,\\
	%\label{Extension formula for Ej Theorem}
	%&&\widetilde{E}_{j}(x',x_{3}):=
	%		 E_{j}(x',-x_{3})+2i\lambda \int\limits\limits_{0}^{-x_{3}}e^{-i\lambda (s+x_3)}E_{j}(x',s)\D s \\ \nonumber
	% && +\frac{2\lambda^{2}}{k^{2}-\lambda^{2}}\int\limits\limits_{0}^{-x_{3}}e^{-i\lambda (s+x_3)}\PD_{j}E_{3}(x',s)\D s
	% -\frac{2k^{2}}{k^{2}-\lambda^{2}}\int\limits\limits_{0}^{-x_{3}}e^{\frac{k^{2}}{i\lambda}(s+x_3)}\PD_{j}E_{3}(x',s)\D s
	%		\en
	%	for $j=1,2$ in $\Omega^{-}$.
\end{theorem}
	Obviously, the extension operator $(\mathcal{D}{E})_{3}$ only relies on $E_3$ in $\Omega^+$, whereas $(\mathcal{D}{E})_{j}$ ($j=1,2$) depends on both $E_j$ and $E_3$. {The formula given by  \eqref{Extension formula} is a \textquoteleft non-point-to-point' reflection formula which is in contrast with the \textquoteleft point-to-point' reflection formula for the Maxwell's equations with the Dirichlet boundary condition (see \cite{Liu_Yamamoto_Zou_Reflection principle_sound-soft_Maxwell}). As $\lambda\rightarrow\infty$, the impedance boundary condition will be reduced to the Dirichlet boundary condition $\nu\times E=0$ on $\Pi$.
	It is easy to observe that $(\mathcal{D}{E})_{3}\rightarrow E_3$ and by applying integration by part that $(\mathcal{D}{E})_{j}\rightarrow-E_j$ for $j=1,2$ as $\lambda$ tends to infinity. Hence, the reflection formula \eqref{Extension formula} becomes $\widetilde{E}(x)=-R_\Pi E(R_\Pi x)$ for $x\in \Omega$, which is valid for any symmetric domain with respect to $\Pi=\{x: x_3=0\}$.
}
	Before going to the proof of Theorem \ref{Main theorem reflection}, we first state the reflection principle for the Helmholtz equation with an impedance  boundary condition. The result in the following Lemma \ref{Lemma reflection principle Helmholtz} has already been proved in \cite{Diaz_Ludford-Reflection principle}. %but since we have shortened and simplified the arguments of \cite{Diaz_Ludford-Reflection principle}, we prefer to provide its proof here.
\begin{lemma}\cite{Diaz_Ludford-Reflection principle}\label{Lemma reflection principle Helmholtz}
	Let $\Omega,\Pi$,$\gamma$ and $\Omega^{\pm}$ be defined as in Theorem \ref{Main theorem reflection}. If $u$ is a solution to the boundary value problem of the Helmholtz equation
	\begin{equation}\label{Helmholtz equation u}
	\Delta u+k^{2}u=0\quad \mbox{in}\ \quad\Omega^{+},\qquad
	\PD_{\nu}u+i\lambda u=0\quad \mbox{on}\quad\gamma,
	\end{equation}
	then $u$ can be extended from $\Omega^{+}$ to $\Omega$ as a solution to the Helmholtz equation, with the extended solution $\widetilde{u}$  given by the formula $\widetilde{u}:=u$ in $\Omega^{+}\cup \gamma$ and
	\begin{equation}\label{u tilde explicit}
	\widetilde{u}(x):=
	u(x_{1},x_{2},-x_{3})+2i\lambda e^{-i\lambda x_{3}}\int\limits\limits_{0}^{-x_{3}}e^{-i\lambda s}u(x_{1},x_{2},s)\D s \quad\mbox{in}\quad\Omega^{-}.
	\end{equation}
\end{lemma}
	Our proof of Theorem
\ref{Main theorem reflection} is essentially motivated by the proof of Lemma \ref{Lemma reflection principle Helmholtz}.
\subsection{Proof of Theorem \ref{Main theorem reflection}}
%\textbf{Proof of Theorem \ref{Main theorem reflection}.}
From Equation \eqref{Main equation}, we deduce that $E$ is a solution to
\begin{align}
\nabla\times\lb \nabla\times E\rb -k^{2}E=0\ \ \mbox{in}\ \Omega^{+},\quad
\textbf{e}_3\times\lb \nabla\times E\rb +i\lambda \textbf{e}_3\times\lb\textbf{e}_3\times E\rb =0 \ \ \mbox{on}\ \gamma.
\end{align}
	Define the vector function $F$ and the scalar function $V$ by %$F:=\nu\times\lb \nabla\times E\rb +i\lambda\nu\times\lb \nu\times E\rb$ then
\begin{align}\nonumber
F&:=\textbf{e}_3\times\lb \nabla\times E\rb +i\lambda\textbf{e}_3\times\lb \textbf{e}_3\times E\rb=:\lb F_{1},F_{2},0\rb^{T}\quad\mbox{in}\quad\Omega^+,\\ \label{F}
&F_1\ = \PD_{1}E_{3}-\PD_{3}E_{1}-i\lambda E_{1},\quad F_2=\PD_{2}E_{3}-\PD_{3}E_{2}-i\lambda E_{2},\\
%\end{align}
%and define $V$ as
%\begin{align*}
\nonumber V&:=\frac{\PD_{1}F_{1}+\PD_{2}F_{2}}{i\lambda}=\frac{\PD_{1}^{2}E_{3}-\PD_{31}^{2}E_{1}-i\lambda \PD_{1}E_{1}+\PD_{2}^{2}E_{3}-\PD_{32}^{2}E_{2}-i\lambda\PD_{2}E_{2}}{i\lambda}.
\end{align}
	Now using $\Delta E+k^{2}E=0$ and $\nabla \cdot E=0$ in $\Omega^{+}$,  we have
\[V=\frac{-\PD_{3}^{2}E_{3}-k^{2}E_{3}-\PD_{3}\lb \PD_{1}E_{1}+\PD_{2}E_{2}\rb-i\lambda \lb \PD_{1}E_{1}+\PD_{2}E_{2}\rb}{i\lambda}=\PD_{3}E_{3}-\frac{k^{2}}{i\lambda}E_{3}. \]
Since $F_{1}=F_{2}=0$ on $\gamma\subseteq\{x\in\Rb^{3}:\ x_{3}=0\}$, we get $\PD_{1}F_{1}=\PD_{2}F_{2}=0$ on $\gamma$ and thus
\begin{align}
\Delta E_{3}+k^{2}E_{3}=0\quad\mbox{in}\quad\Omega^{+},\quad
\PD_{3}E_{3}-k^{2}/(i\lambda)\; E_{3}=0\quad\mbox{on}\quad \gamma.
\end{align}
Applying Lemma \ref{Lemma reflection principle Helmholtz}, we can extend $E_{3}$ from $\Omega^{+}$ to $\Omega$ by
$\widetilde{E}_{3}:=E_3$ in $\Omega^{+}\cup\gamma$ and
\begin{align*}
\widetilde{E}_{3}(x):=
E_{3}(x',-x_{3})-\frac{2k^{2}}
{i\lambda}e^{\frac{k^{2}}{i\lambda}x_{3}}\int\limits\limits_{0}^{-x_{3}}e^{\frac{k^{2}}{i\lambda}s}E_{3}(x',s)\D s\ \mbox{in}\ \Omega^{-},
\end{align*}
which gives the extension formula for $E_{3}$.
	To find the extension formula for $E_{j}$ ($j=1,2$), we observe that
$F_{j}$ ($j=1,2$) given by (\ref{F})
satisfy the Helmholtz equation with the Dirichlet boundary condition,
\begin{align*}
\Delta F_{j}+k^{2}F_{j}=0\ \mbox{in} \ \Omega^{+},\qquad
F_{j}=0\ \mbox{on}\ \gamma.
\end{align*}
Applying the reflection principle with the Dirichlet boundary condition (see \cite{Diaz_Ludford-Reflection principle}), we can extend $F_j$ through
$\widetilde{F}_{j}:=F_j$ in $\Omega^{+}\cup \gamma$ and $\widetilde{F}_{j}(x):=-F_j(x',-x_3)$ in $\Omega^-$.
	As done for the Helmholtz equation, we will derive the extension formula for $E_{j}$ for $j=1,2$ by considering the boundary value problem of the ordinary differential equation (cf. (\ref{F}))
\begin{align*}
\PD_{3}\widetilde{E}_{j}+i\lambda \widetilde{E}_{j}-\PD_{j}\widetilde{E}_{3}=-\widetilde{F}_{j}\quad\mbox{in}\quad
\Omega,\qquad \widetilde{E}_{j}=E_j\quad\mbox{on}\quad\gamma,
\end{align*}
where $\widetilde{E}_{j}$ ($j=1,2$) denote the extended functions.
Multiplying the above equation by $e^{i\lambda x_{3}}$ and integrating between $0$ to $x_{3}$, we have
\begin{align*}
\int\limits\limits_{0}^{x_{3}}e^{i\lambda s}\lb \PD_{s}\widetilde{E}_{j}(x',s)+i\lambda \widetilde{E}_{j}(x',s)\rb \D s-\int\limits\limits_{0}^{x_{3}}e^{i\lambda s}\PD_{j}\widetilde{E}_{3}(x',s)\D s=-\int\limits\limits_{0}^{x_{3}}e^{i\lambda s}\widetilde{F}_{j}(x',s)\D s,
\end{align*}
which gives us
\begin{align*}
\widetilde{E}_{j}(x)=e^{-i\lambda x_{3}}E_{j}(x',0)+\int\limits\limits_{0}^{x_{3}}e^{i\lambda (s-x_3)}\PD_{j}\widetilde{E}_{3}(x',s)\D s-\int\limits\limits_{0}^{x_{3}}e^{i\lambda (s-x_3)}\widetilde{F}_{j}(x',s)\D s.
\end{align*}
	Since $\widetilde{F}_{j}=F_{j}$ and $\widetilde{E}_{3}=E_{3}$ in $\in\Omega^{+}$,  the above equation can be rewritten as
\begin{align*}
\widetilde{E}_{j}(x)&=e^{-i\lambda x_{3}}E_{j}(x',0)+e^{-i\lambda x_{3}}\int\limits\limits_{0}^{x_{3}}e^{i\lambda s}\PD_{j}{E}_{3}(x',s)\D s\\
&\ \ \ \ -e^{-i\lambda x_{3}}\int\limits\limits_{0}^{x_{3}}e^{i\lambda s}\lb \PD_{j}E_{3}(x',s)-\PD_{s}E_{j}(x',s)-i\lambda E_{j}(x',s)\rb\D s
\end{align*}
which can be proved to be identical with $E_j$ in $\Omega^+$ by applying the integration by parts.
Next, we want to simplify the expression of $\widetilde{E}_{j}(x)$ in $\Omega^{-}$. Using the expression for $\widetilde{F}_{j}$  and $\widetilde{E}_{3}$ (see \eqref{E3 tilde extension Theorem}), we obtain
	\ben
\widetilde{E}_{j}(x)%=e^{-i\lambda x_{3}}E_{j}(x',0)+e^{-i\lambda x_{3}}\int\limits\limits_{0}^{x_{3}}e^{i\lambda s}\lb \PD_{j}E_{3}(x',-s)-\frac{2k^{2}}{i\lambda}e^{\frac{k^{2}}{i\lambda}s}\int\limits\limits_{0}^{-s}e^{\frac{k^{2}}{i\lambda}t}\PD_{j}E_{3}(x',t)\D t\rb\D s\\
%&\ \ \ \ +e^{-i\lambda x_{3}}\int\limits\limits_{0}^{x_{3}}e^{i\lambda s}F_{j}(x',-s)\D s\\
&=&e^{-i\lambda x_{3}}E_{j}(x',0)+\int\limits\limits_{0}^{x_{3}}e^{i\lambda (s-x_3)}\lb \PD_{j}E_{3}(x',-s)+\PD_{s} E_{j}(x',-s)-i\lambda E_{j}(x',-s)\rb\D s\\
&+&\int\limits\limits_{0}^{x_{3}}e^{i\lambda (s-x_3)} \PD_{j}E_{3}(x',-s)\D s-\frac{2k^{2}}{i\lambda} \int\limits\limits_{0}^{x_{3}}e^{i\lambda (s-x_3)} \lb \int\limits\limits_{0}^{-s}e^{\frac{k^{2}}{i\lambda}(t+s)}\PD_{j}E_{3}(x',t)\D t\rb\D s.
\enn
This gives
	\begin{align*}
\widetilde{E}_{j}(x)%&=e^{-i\lambda x_{3}}E_{j}(x',0)-2e^{-i\lambda x_{3}}\int\limits\limits_{0}^{-x_{3}}e^{-i\lambda s} \PD_{j}E_{3}(x',s)\D s-\frac{2k^{2}}{i\lambda}e^{-i\lambda x_{3}}\int\limits\limits_{0}^{x_{3}}e^{\frac{k^{2}-\lambda^{2}}{i\lambda}s}\int\limits\limits_{0}^{-s}e^{\frac{k^{2}}{i\lambda}t}\PD_{j}E_{3}(x',t)\D t\D s\\
%&\ \ \ \ + E_{j}(x',-x_{3})-e^{-i\lambda x_{3}}E_{j}(x',0)+2i\lambda e^{-i\lambda x_{3}}\int\limits\limits_{0}^{x_{3}}e^{i\lambda s}E_{j}(x',-s)\D s\\
&= E_{j}(x',-x_{3})+2i\lambda \int\limits\limits_{0}^{-x_{3}}e^{-i\lambda (s+x_3)}E_{j}(x',s)\D s-2\int\limits\limits_{0}^{-x_{3}}e^{-i\lambda (s+x_3)}\PD_{j}E_{3}(x',s)\D s\\
&\ \ \ \  \ +\frac{2k^{2}}{i\lambda}\int\limits\limits_{0}^{x_{3}}e^{i\lambda (s- x_{3})}
\int\limits\limits_{0}^{s}e^{-\frac{k^{2}}{i\lambda}(t-s)}\PD_{j}E_{3}(x',-t)\D t\D s.
\end{align*}
	Changing the order of integration in the last term of the previous equation, we obtain
\begin{align*}
\begin{aligned}
\widetilde{E}_{j}(x)&=E_{j}(x',-x_{3})+2i\lambda \int\limits\limits_{0}^{-x_{3}}e^{-i\lambda (s+x_3)}E_{j}(x',s)\D s-2\int\limits\limits_{0}^{-x_{3}}e^{-i\lambda (s+x_3)}\PD_{j}E_{3}(x',s)\D s\\
&\  +\frac{2k^{2}}{i\lambda}e^{-i\lambda x_{3}}\int\limits\limits_{t=0}^{t=x_{3}}e^{-\frac{k^{2}}{i\lambda}t}\PD_{j}E_{3}(x',-t)\int\limits\limits_{s=t}^{s=x_{3}}e^{\frac{k^{2}-\lambda^{2}}{i\lambda}s}\D s\D t,\\
& =E_{j}(x',-x_{3})+2i\lambda \int\limits\limits_{0}^{-x_{3}}e^{-i\lambda (s+x_3)}E_{j}(x',s)\D s-2\int\limits\limits_{0}^{-x_{3}}e^{-i\lambda (s+x_3)}\PD_{j}E_{3}(x',s)\D s\\
&+\frac{2k^{2}}{k^{2}-\lambda^{2}}\int\limits\limits_{0}^{x_{3}}
e^{-\frac{k^{2}}{i\lambda}(s-x_3)}\PD_{j}E_{3}(x',-s)\D s-\frac{2k^{2}}{k^{2}-\lambda^{2}}\int\limits\limits_{0}^{x_{3}}e^{i\lambda (s-x_3)}\PD_{j}E_{3}(x',-s)\D s.
%
%&= E_{j}(x',-x_{3})+2i\lambda e^{-i\lambda x_{3}}\int\limits\limits_{0}^{-x_{3}}e^{-i\lambda s}E_{j}(x',s)\D s-2e^{-i\lambda x_{3}}\int\limits\limits_{0}^{-x_{3}}e^{-i\lambda s}\PD_{j}E_{3}(x',s)\D s\\
%&\ \ \ \ -\frac{2k^{2}}{k^{2}-\lambda^{2}}e^{\frac{k^{2}}{i\lambda}x_{3}}\int\limits\limits_{0}^{-x_{3}}e^{\frac{k^{2}}{i\lambda}s}\PD_{j}E_{3}(x',s)\D s+\frac{2k^{2}}{k^{2}-\lambda^{2}}e^{-i\lambda x_{3}}\int\limits\limits_{0}^{-x_{3}}e^{-i\lambda s}\PD_{j}E_{3}(x',s)\D s.
\end{aligned}
\end{align*}
	After combining  similar terms, we get
\begin{align*}
\begin{aligned}
\widetilde{E}_{j}(x)&=E_{j}(x',-x_{3})+2i\lambda \int\limits\limits_{0}^{-x_{3}}e^{-i\lambda (s+x_3)}E_{j}(x',s)\D s\\
&+\frac{2\lambda^{2}}{k^{2}-\lambda^{2}}\int\limits\limits_{0}^{-x_{3}}e^{-i\lambda (s+x_3)}\PD_{j}E_{3}(x',s)\D s-\frac{2k^{2}}{k^{2}-\lambda^{2}}\int\limits\limits_{0}^{-x_{3}}
e^{\frac{k^{2}}{i\lambda}(s+x_3)}\PD_{j}E_{3}(x',s)\D s.
\end{aligned}
\end{align*}
This proves Equation \eqref{Extension formula for Ej Theorem}.

 Remark that the right hand side of $\widetilde{E}_{j}$ ($j=1,2$) depends on both $E_j$ and $E_3$ in $\Omega^+$.
	In order to show that $\widetilde{E}_{j}$ given by \eqref{Extension formula for Ej Theorem} and \eqref{E3 tilde extension Theorem} are indeed the required extension formula for $E_{j}$, we need to verify that $\Delta \widetilde{E}_{j}+k^{2}\widetilde{E}_{j}=0$ and  $\nabla\cdot \widetilde{E}=0$ in $\Omega$. For this purpose, we shall proceed with the following three steps.

Step 1. Prove that the Cauchy data of $\widetilde{E}_{j}$ taking from $\Omega^\pm$ are identical on $\gamma$. By Lemma \ref{Lemma reflection principle Helmholtz}, this is true for the third component  $\widetilde{E}_{3}$. On the other hand,
it is  clear from Equation \eqref{Extension formula for Ej Theorem} that $\widetilde{E}_{j}$ ($j=1,2$) are continuous functions in $\Omega$. Therefore, we only need to show that $\PD_{3}^{+}\widetilde{E}_{j}=\PD_{3}^{-}\widetilde{E}_{j}$ on $\gamma$, $j=1,2$. Simple calculations show that
\begin{align*}%\label{First derivative of Ej tilde}
\begin{aligned}
\PD_{3}\widetilde{E}_{j}(x)=
\begin{cases}
\PD_{3}E_{j}(x',x_{3})\quad \mbox{in}\ \Omega^{+},\\
-\PD_{3}E_{j}(x',-x_{3})+2\lambda^{2}e^{-i\lambda x_{3}}\int\limits\limits\limits_{0}^{-x_{3}}e^{-i\lambda s}E_{j}(x',s)\D s-2i\lambda E_{j}(x',-x_{3})\\
\ \ -\frac{2i\lambda^{3}}{k^{2}-\lambda^{2}}e^{-i\lambda x_{3}}\int\limits\limits\limits_{0}^{-x_{3}}e^{-i\lambda s}\PD_{j}E_{3}(x',s)\D s -\frac{2\lambda^{2}}{k^{2}-\lambda^{2}}\PD_{j}E_{3}(x',-x_{3})\\
-\frac{2k^{4}}{i\lambda\lb k^{2}-\lambda^{2}\rb}
e^{\frac{k^{2}}{i\lambda}x_{3}}
\int\limits\limits\limits_{0}^{-x_{3}}e^{\frac{k^{2}}{i\lambda}s}\PD_{j}E_{3}(x',s)\D s+\frac{2k^{2}}{k^{2}-\lambda^{2}}\PD_{j}E_{3}(x',-x_{3})\ \mbox{in}\ \Omega^{-},
\end{cases}
\end{aligned}
\end{align*}
	from which it follows that
\begin{align*}
\PD_{3}^-\widetilde{E}_{j}(x',0)=-\PD_{3}^+E_{j}(x',0)-2i\lambda E_{j}(x',0)
+\frac{2(k^{2}-\lambda^2)}{k^{2}-\lambda^{2}}\PD_{j}E_{3}(x',0).
\end{align*}
Recalling the relation $\PD_{j}E_{3}(x',0)-\PD_{3}^+E_{j}(x',0)-i\lambda E_{j}(x',0)=0$ for $j=1,2$, we get from the previous equation that $\PD_{3}^{+}E_{j}=\PD_{3}^{-}\widetilde{E}_{j}$ on $\gamma$.

	Step 2. Prove that $\Delta \widetilde{E}_{j}+k^{2}\widetilde{E}_{j}=0$ in $\Omega$ for $j=1,2,3$. In view of Step 1, it suffices to verify that $\Delta \widetilde{E}_{j}+k^{2}\widetilde{E}_{j}=0$ in $\Omega^-$ for $j=1,2$.
		From Equation \eqref{Extension formula for Ej Theorem}, we have
	\begin{align}\label{Laplacian Ej tilde in Omega negative}
	\Delta \widetilde{E}_{j}(x)=\Delta E_{j}(x',-x_{3})+I_1+I_2+I_3,\quad x\in \Omega^{-},
	\end{align}
	where, for some fixed $j=1$ or $j=2$,
	\begin{align*}
	&I_1:=2i\lambda\Delta\lb  e^{-i\lambda x_{3}}\int\limits\limits_{0}^{-x_{3}}e^{-i\lambda s}E_{j}(x',s)\D s \rb\\
	&I_2:=\frac{2\lambda^{2}}{k^{2}-\lambda^{2}}\Delta\lb e^{-i\lambda x_{3}}\int\limits\limits_{0}^{-x_{3}}e^{-i\lambda s}\PD_{j}E_{3}(x',s)\D s\rb,\\
	&	I_3:=   -\frac{2k^{2}}{k^{2}-\lambda^{2}}\Delta\lb e^{\frac{k^{2}}{i\lambda}x_{3}}\int\limits\limits_{0}^{-x_{3}}e^{\frac{k^{2}}{i\lambda}s}\PD_{j}E_{3}(x',s)\D s\rb.
	\end{align*}
	Using $\Delta E_{j}+k^{2}E_{j}=0$ for $j=1,2,3$ in $\Omega^{+}$ and applying integration by parts, the three terms $I_j$ ($j=1,2,3$) can be calculated as follows:
		\ben
	I_{1}&=-2i\lambda \PD_{3}E_{j}(x',-x_{3})+2i\lambda e^{-i\lambda x_{3}}\PD_{3}E_{j}(x',0)+2\lambda^{2}E_{j}(x',-x_{3}) -2\lambda^{2}e^{-i\lambda x_{3}}E_{j}(x',0)\\
	& +2i\lambda^{3}e^{-i\lambda x_{3}}\int\limits\limits_{0}^{-x_{3}}e^{-i\lambda s}E_{j}(x',s)\D s-2i\lambda k^{2}e^{-i\lambda x_{3}}\int\limits\limits_{0}^{-x_{3}}e^{-i\lambda s}E_{j}(x',s)\D s\\
	&-2i\lambda^{3}e^{-i\lambda x_{3}}\int\limits\limits_{0}^{-x_{3}}e^{-i\lambda s}E_{j}(x',s)\D s
	-2\lambda^{2}E_{j}(x',-x_{3})+2i\lambda \PD_{3}E_{j}(x',-x_{3}),
	\enn
		\ben
	I_{2}&=-\frac{2\lambda^{2}}{k^{2}-\lambda^{2}}\PD_{3}\PD_{j}E_{3}(x',-x_{3}) +\frac{2\lambda^{2}}{k^{2}-\lambda^{2}}e^{-i\lambda x_{3}}\PD_{3}\PD_{j}E_{3}(x',0)  -\frac{2i\lambda^{3}}{k^{2}-\lambda^{2}}\PD_{j}E_{3}(x',-x_{3})\\
	&\ \  +\frac{2i\lambda^{3}}{k^{2}-\lambda^{2}}e^{-i\lambda x_{3}}\PD_{j}E_{3}(x',0) +\frac{2\lambda^{4}}{k^{2}-\lambda^{2}}e^{-i\lambda x_{3}}\int\limits\limits_{0}^{-x_{3}}e^{-i\lambda s}\PD_{j}E_{3}(x',s)\D s\\
	&\ \ -\frac{2\lambda^{2}k^{2}}{k^{2}-\lambda^{2}}e^{-i\lambda x_{3}}\int\limits\limits_{0}^{-x_{3}}e^{-i\lambda s}  \PD_{j}E_{3}(x',s)\D s -\frac{2\lambda^{4}}{k^{2}-\lambda^{2}}e^{-i\lambda x_{3}}\int\limits\limits_{0}^{-x_{3}}e^{-i\lambda s}\PD_{j}E_{3}(x',s)\D s\\
	&\ \ +\frac{2i\lambda^{3}}{k^{2}-\lambda^{2}}\PD_{j}E_{3}(x',-x_{3}) +\frac{2\lambda^{2}}{k^{2}-\lambda^{2}}\PD_{3}\PD_{j}E_{3}(x',-x_{3}),
	\enn
		\ben
	&I_{3}=\frac{2k^{2}}{k^{2}-\lambda^{2}}\PD_{3}\PD_{j}E_{3}(x',-x_{3})  -\frac{2k^{2}}{k^{2}-\lambda^{2}}e^{\frac{k^{2}}{i\lambda} x_{3}}\PD_{3}\PD_{j}E_{3}(x',0)+\frac{2ik^{4}}{\lambda\lb k^{2}-\lambda^{2}\rb}\PD_{j}E_{3}(x',-x_{3})\\
	&\ \ -\frac{2ik^{4}}{\lambda\lb k^{2}-\lambda^{2}\rb}e^{\frac{k^{2}}{i\lambda}x_{3}}\PD_{j}E_{3}(x',0) -\frac{2k^{6}}{\lambda^{2}\lb k^{2}-\lambda^{2}\rb}e^{\frac{k^{2}}{i\lambda}x_{3}}\int\limits\limits_{0}^{-x_{3}}e^{\frac{k^{2}}{i\lambda}s}\PD_{j}E_{3}(x',s)\D s\\
	&\ \ +\frac{2k^{4}}{k^{2}-\lambda^{2}}e^{\frac{k^{2}}{i\lambda}x_{3}}\int\limits\limits_{0}^{-x_{3}}e^{\frac{k^{2}}{i\lambda}s}\PD_{j}E_{3}(x',s)\D s  +\frac{2k^{6}}{\lambda^{2}\lb k^{2}-\lambda^{2}\rb}e^{\frac{k^{2}}{i\lambda}x_{3}}\int\limits\limits_{0}^{-x_{3}}e^{\frac{k^{2}}{i\lambda}s}\PD_{j}E_{3}(x',s)\D s\\
	&\ \  -\frac{2i k^{4}}{\lambda\lb k^{2}-\lambda^{2}\rb} \PD_{j}E_{3}(x',-x_{3})-\frac{2k^{2}}{k^{2}-\lambda^{2}}\PD_{3}\PD_{j}E_{3}(x',-x_{3}).
	\enn
		Using again the Helmholtz equation $\Delta E_{j}+k^{2}E_{j}=0$ in $\Omega^{+}$ and inserting expressions of $I_1$, $I_2$ and $I_3$ into Equation \eqref{Laplacian Ej tilde in Omega negative}, we get
	\ben
	&&\Delta \widetilde{E}_{j}(x)
		%&=-k^{2}E_{j}(x',-x_{3})+2i\lambda e^{-i\lambda x_{3}}\PD_{3}E_{j}(x',0) -2\lambda^{2}e^{-i\lambda x_{3}}E_{j}(x',0) -2i\lambda k^{2}e^{-i\lambda x_{3}}\int\limits\limits_{0}^{-x_{3}}e^{-i\lambda s}E_{j}(x',s)\D s\\
	%&\ \ \ \ +\frac{2\lambda^{2}}{k^{2}-\lambda^{2}}e^{-i\lambda x_{3}}\PD_{3}\PD_{j}E_{3}(x',0) +\frac{2i\lambda^{3}}{k^{2}-\lambda^{2}}e^{-i\lambda x_{3}}\PD_{j}E_{3}(x',0) -\frac{2\lambda^{2}k^{2}}{k^{2}-\lambda^{2}}e^{-i\lambda x_{3}}\int\limits\limits_{0}^{-x_{3}}e^{-i\lambda s}\PD_{j}E_{3}(x',s)\D s\\
	%&\ \ \ \ -\frac{2k^{2}}{k^{2}-\lambda^{2}}e^{\frac{k^{2}}{i\lambda} x_{3}}\PD_{3}\PD_{j}E_{3}(x',0) -\frac{2ik^{4}}{\lambda\lb k^{2}-\lambda^{2}\rb}e^{\frac{k^{2}}{i\lambda}x_{3}}\PD_{j}E_{3}(x',0)+\frac{2k^{4}}{k^{2}-\lambda^{2}}e^{\frac{k^{2}}{i\lambda}x_{3}}\int\limits\limits_{0}^{-x_{3}}e^{\frac{k^{2}}{i\lambda}s}\PD_{j}E_{3}(x',s)\D s.
	=-k^{2}\Bigg (  E_{j}(x',-x_{3}) +2i\lambda  \int\limits\limits_{0}^{-x_{3}}e^{-i\lambda (s+x_3)}E_{j}(x',s)\D s\\
	&&+\frac{2\lambda^{2}}{k^{2}-\lambda^{2}}\int\limits\limits_{0}^{-x_{3}}e^{-i\lambda (s+x_3)}\PD_{j}E_{3}(x',s)\D s
	-\frac{2k^{2}}{k^{2}-\lambda^{2}}e^{\frac{k^{2}}{i\lambda}x_{3}}\int\limits\limits_{0}^{-x_{3}}e^{\frac{k^{2}}{i\lambda}s}\PD_{j}E_{3}(x',s)\D s       \Bigg)\\
	&&+2i\lambda e^{-i\lambda x_{3}}\Big( \PD_{3}E_{j}(x',0)+i\lambda E_{j}(x',0)\Big)
	+\frac{2\lambda^{2}}{k^{2}-\lambda^{2}}e^{-i\lambda x_{3}}\PD_{j}\Big( \PD_{3}E_{3}(x',0)+i\lambda \PD_{j}E_{j}(x',0)\Big) \\
	&&-\frac{2k^{2}}{k^{2}-\lambda^{2}}e^{\frac{k^{2}}{i\lambda}x_{3}}\PD_{j}\lb \PD_{3}E_{3}(x',0)-\frac{k^{2}}{i\lambda}E_{3}(x',0)\rb.
	\enn
		This together with Equation \eqref{Extension formula for Ej Theorem} and the following boundary conditions
	\begin{align*}
	\PD_{3}E_{3}(x',0)-\frac{k^{2}}{i\lambda}E_{3}(x',0)=0, \quad \PD_{j}E_{3}(x',0)-\PD_{3}E_{j}(x',0)-i\lambda E_{j}(x',0)=0
	\end{align*}
	leads to the relation $\Delta\widetilde{E}_{j}+k^{2}\widetilde{E}_{j}=0$ in $\Omega^{-}$.

		Step 3. Prove that $\nabla\cdot \widetilde{E}=0$ in $\Omega$. It follows from Step 1 that $\tilde{E}\in C^1(\Omega)$. Hence, we only need to show the divergence-free condition in $\Omega^-$. For $x\in\Omega^{-}$, we see
			\begin{eqnarray*}
			\nabla\cdot \widetilde{E}(x)&=& \PD_{1}E_{1}(x',-x_{3})+\PD_{2}E_{2}(x',-x_{3})-\PD_{3}E_{3}(x',-x_{3}) \\ &&+2i\lambda \int\limits\limits_{0}^{-x_{3}}e^{-i\lambda (s+x_3)}\lb \PD_{1}E_{1}+\PD_{2}E_{2}\rb (x',s)\D s\\
			&& +\frac{2\lambda^{2}}{k^{2}-\lambda^{2}}\int\limits\limits_{0}^{-x_{3}}e^{-i\lambda (s+x_3)}\lb \PD_{1}^{2}E_{3}+\PD_{2}^{2}E_{3}\rb (x',s)\D s \\ &&-\frac{2k^{2}}{k^{2}-\lambda^{2}}\int\limits\limits_{0}^{-x_{3}}
			e^{\frac{k^{2}}{i\lambda}(s+x_3)}\lb \PD_{1}^{2}E_{3}+\PD_{2}^{2}E_{3}\rb(x',s)\D s\\
			&&  +\frac{2k^{4}}{\lambda^{2}}\int\limits\limits_{0}^{-x_{3}}e^{\frac{k^{2}}{i\lambda}(s+x_3)}E_{3}(x',s)\D s+\frac{2k^{2}}{i\lambda}E_{3}(x',-x_{3}).
		\end{eqnarray*}
		Now using $\nabla \cdot E=0$ and $\Delta E_{j}+k^{2}E_{j}=0$ for $1\leq j\leq 3$ in $\Omega^{+}$, we have
		\begin{align}\label{Divergence E tilde}
		\begin{aligned}
		&\nabla\cdot \widetilde{E}(x)=-2\PD_{3}E_{3}(x',-x_{3})-\underbrace{2i\lambda \int\limits\limits_{0}^{-x_{3}}e^{-i\lambda (s+x_3)}\PD_{s}E_{3}(x',s)\D s}_{J_{1}}\\
		& -\underbrace{\frac{2\lambda^{2}}{k^{2}-\lambda^{2}}\int\limits\limits_{0}^{-x_{3}}
			e^{-i\lambda (s+x_3)}\PD_{s}^{2}E_{3} (x',s)\D s}_{J_{2}}-\frac{2\lambda^{2}k^{2}}{k^{2}-\lambda^{2}}\int\limits\limits_{0}^{-x_{3}}
		e^{-i\lambda (s+x_3)}E_{3} (x',s)\D s\\
		&+\underbrace{\frac{2k^{2}}{k^{2}-\lambda^{2}}\int\limits\limits_{0}^{-x_{3}}
			e^{\frac{k^{2}}{i\lambda}(s+x_3)}\PD_{s}^{2}E_{3}(x',s)\D s}_{J_{3}} +\frac{2k^{4}}{k^{2}-\lambda^{2}}\int\limits\limits_{0}^{-x_{3}}e^{\frac{k^{2}}{i\lambda}(s+x_3)}
		E_{3}(x',s)\D s\\
		& +\frac{2k^{4}}{\lambda^{2}}\int\limits\limits_{0}^{-x_{3}}e^{\frac{k^{2}}{i\lambda}(s+x_3)}
		E_{3}(x',s)\D s+\frac{2k^{2}}{i\lambda}E_{3}(x',-x_{3}).
		\end{aligned}
		\end{align}
		Using integration by parts, we can rewrite $J_{1},J_{2}$ and $J_{3}$ as
		\ben
		\label{J2 in simplified form}
		J_{2}&=&\frac{2\lambda^{2}}{k^{2}-\lambda^{2}}\PD_{3}E_{3}(x',-x_{3}) -\frac{2\lambda^{2}}{k^{2}-\lambda^{2}}e^{-i\lambda x_{3}}\PD_{3}E_{3}(x',0)+\frac{2i\lambda^{3}}{k^{2}-\lambda^{2}}E_{3}(x',-x_{3})\\
		&&
		-\frac{2i\lambda^{3}}{k^{2}-\lambda^{2}}e^{-i\lambda x_{3}}E_{3}(x',0)
		-\frac{2\lambda^{4}}{k^{2}-\lambda^{2}}e^{-i\lambda x_{3}}\int\limits\limits_{0}^{-x_{3}}e^{-i\lambda s}E_{3}(x',s)\D s\\
		&& +\frac{2\lambda^{2}k^{2}}{k^{2}-\lambda^{2}}e^{-i\lambda x_{3}}\int\limits\limits_{0}^{-x_{3}}e^{-i\lambda s}E_{3} (x',s)\D s,\\
		\label{J1 in simplified form}
		J_{1}&=&2i\lambda E_{3}(x',-x_{3})-2i\lambda e^{-i\lambda x_{3}}E_{3}(x',0)-2\lambda^{2}e^{-i\lambda x_{3}}\int\limits\limits_{0}^{-x_{3}}e^{-i\lambda s}E_{3}(x',s)\D s,\\
		J_{3}&=&\frac{2k^{2}}{k^{2}-\lambda^{2}}\PD_{3}E_{3}(x',-x_{3}) -\frac{2k^{2}}{k^{2}-\lambda^{2}}e^{\frac{k^{2}}{i\lambda}x_{3}}\PD_{3}E_{3}(x',0)+\frac{2ik^{4}}{\lambda \lb k^{2}-\lambda^{2}\rb}E_{3}(x',-x_{3})\\
		&& -\frac{2ik^{4}}{\lambda \lb k^{2}-\lambda^{2}\rb}e^{\frac{k^{2}}{i\lambda}x_{3}}E_{3}(x',0)
		-\frac{2k^{6}}{\lambda^{2}\lb k^{2}-\lambda^{2}\rb}e^{\frac{k^{2}}{i\lambda}x_{3}}\int\limits\limits_{0}^{-x_{3}}e^{\frac{k^{2}}{i\lambda} s}E_{3}(x',s)\D s.
		\enn
		Inserting them into Equation \eqref{Divergence E tilde}, applying the integration by parts and rearranging terms, we get
		\begin{align*}
		\begin{aligned}
		\nabla\cdot \widetilde{E}(x)&=2e^{-i\lambda x_{3}}\Bigg[\lb i\lambda +\frac{i\lambda^{3}}{k^{2}-\lambda^{2}}\rb E_{3}(x',0)+\frac{\lambda^{2}}{k^{2}-\lambda^{2}} \PD_{3}E_{3}(x',0)\Bigg]\\
		& \ \  +\frac{2k^{2}}{k^{2}-\lambda^{2}}e^{\frac{k^{2}}{i\lambda}x_{3}}\lb -\PD_{3}E_{3}(x',0)+\frac{k^{2}}{i\lambda} E_{3}(x',0)\rb.
		\end{aligned}
		\end{align*}
		Recalling
			$\PD_{3}E_{3}(x',0)-\frac{k^{2}}{i\lambda}E_{3}(x',0)=0$,  we finally get $\nabla\cdot\widetilde{E}=0$ in $\Omega^{-}$.
		
		By far we have proved that the function $\widetilde{E}$ with components given by Equations \eqref{Extension formula}, \eqref{E3 tilde extension Theorem} and  \eqref{Extension formula for Ej Theorem}  is the extension of the solution $E$ of the Maxwell's equations. \hfill$\Box$.
			
		\section{Applications of the Reflection Principle}
		The main purpose of this section is to prove the uniqueness result for recovering convex polyhedral scatterers of impedance type, which was stated in Theorem \ref{Main theorem scattering}  when the incident fields are given by either \eqref{Point source waves}  or \eqref{Herglotz wave fields}. This part also gives a new proof for electromagnetic plane waves.

		Assuming two of such different scatterers generate identical far-field patterns,
		we shall prove via reflection principle that the scattered electric field could be analytically extended into the whole space, which is impossible.  Similar ideas were employed in \cite{FI,Ik2000, Elschner_Hu_Elastic} for proving uniqueness in inverse conductivity and elastic scattering problems. Later we shall remark why our approach cannot be applied to non-convex polyhedral scatterers and compare our arguments with the uniqueness proof of \cite{Cheng2003} in the Helmholtz case. The Corollaries below follow straightforwardly from  Theorem \ref{Main theorem reflection}.
		\begin{corollary}\label{cor3.1}
			Let $(E, H)$ be a solution to the Maxwell's equations (\ref{Main equation}) in $x_3>0$ fulfilling the impedance boundary condition (\ref{Boundary condition}) on $\Pi=\{x\in \mathbb{R}^3: x_3=0\}$. Then $(E, H)$ can be extended from the upper half-space $x_3\geq0$ to the whole space.
			Moreover, the extended electric field $\widetilde{E}$ is given by
			\ben
			\widetilde{E}(x)=\left\{\begin{array}{lll}
				E(x),\ \ &&\mbox{if}\ x_3\geq 0\\
				\mathcal{D}E(x',-x_{3}),\ \ &&\mbox{if}\ x_3<0,
			\end{array}\right.
			\enn
			where $\mathcal{D}$ is the operator defined in Theorem \ref{Main theorem reflection}.
		\end{corollary}
		\begin{corollary}\label{cor3.2}
		Let $\Omega=\Omega^+\cup\gamma\cup\Omega^-$ be the domain defined in Theorem \ref{Main theorem reflection}. Given a subset $D\subset \Omega^-$, suppose that $(E, H)$ is a solution to the Maxwell's equations (\ref{Main equation}) in $\Omega\backslash \overline{D}$ fulfilling the impedance boundary condition (\ref{Boundary condition}) on $\gamma$.
		Then $(E, H)$ can be analytically extended onto $\overline{D}$.
	\end{corollary}	
	The above results will be used in the proof of
Theorem \ref{Main theorem scattering} to be carried out below.	
	\subsection{Proof of Theorem \ref{Main theorem scattering} for electromagnetic Herglotz waves and point source waves}\label{Main theorem}

First, we proceed with the same arguments for electromagnetic plane waves. Assume that there are two different convex polyhedrons $D_1$ and $D_2$ which generate the same electric far-field pattern.
We assume without loss of generality that  there exists a vertex $O$ of $\PD D_{1}$ and a neighborhood $V_{O}$ of $O$ such that $V_{O}\cap \overline{D}_{2}=\emptyset$; see Figure \ref{f1}. Denote by $\Lambda_j\subset \partial D_1$ $(j=1,2\cdots m)$ the $m\geq 3$ convex polygonal faces of $\partial D_{1}$ whose closure meet at $O$ and by $\widetilde{\Pi}_j\supseteq \Lambda_j$ their maximum analytical extension in $\Rb^3\backslash\overline{D}_2$.
	Then we get
\[\nu\times \lb \nabla\times E_{2}\rb+i\lambda\nu\times\lb\nu\times E_{2}\rb=0\quad \mbox{on}\quad \widetilde{\Pi}_j,\quad j=1,2,\cdots,m
\]
due to the analyticity of $E_2$ in the exterior of $\overline{D}_2$.
\begin{figure}[htbp]
	\centering
	\includegraphics[width=5.7cm,height=4cm]{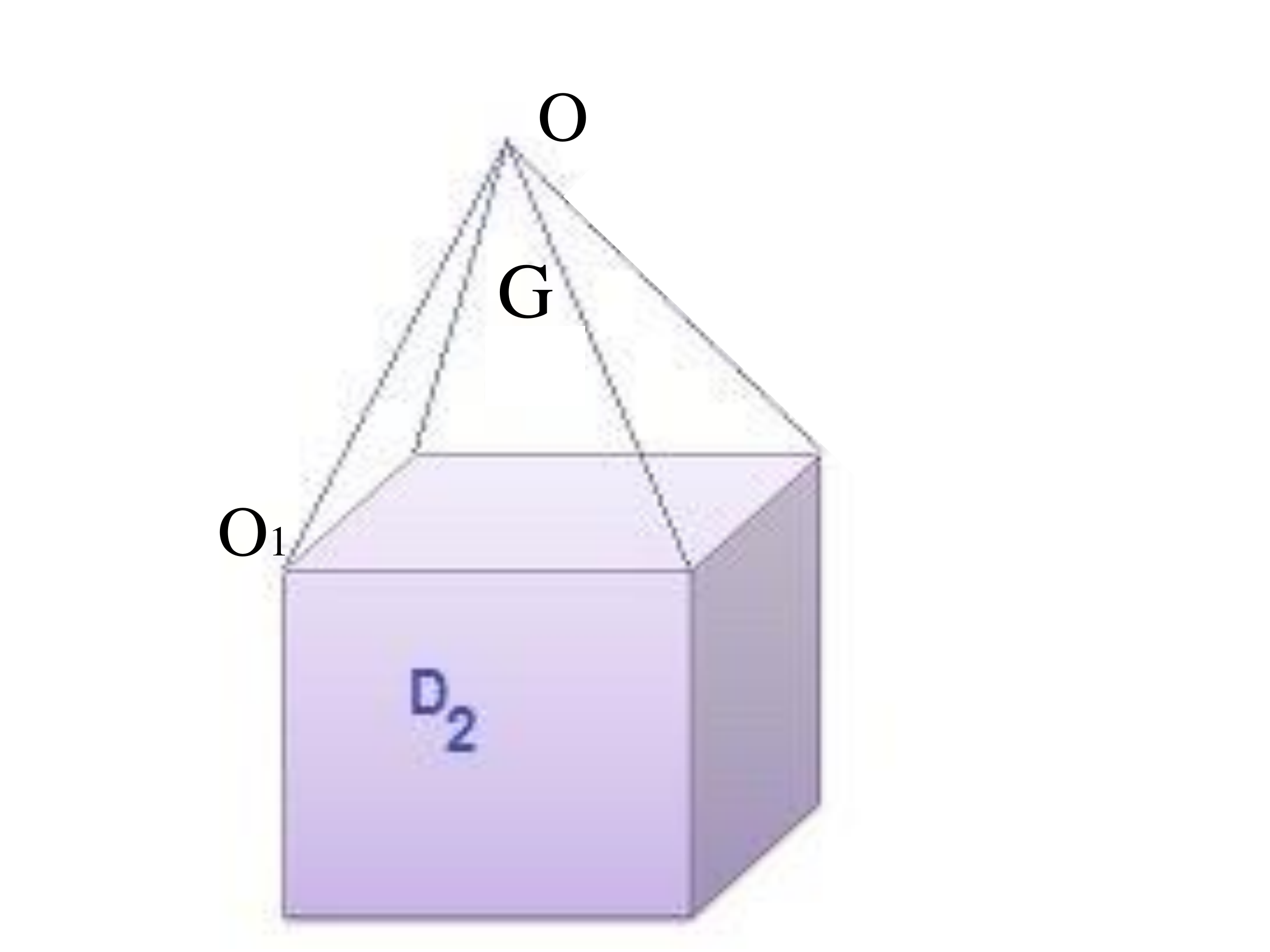}
	\includegraphics[width=6cm,height=4.5cm]{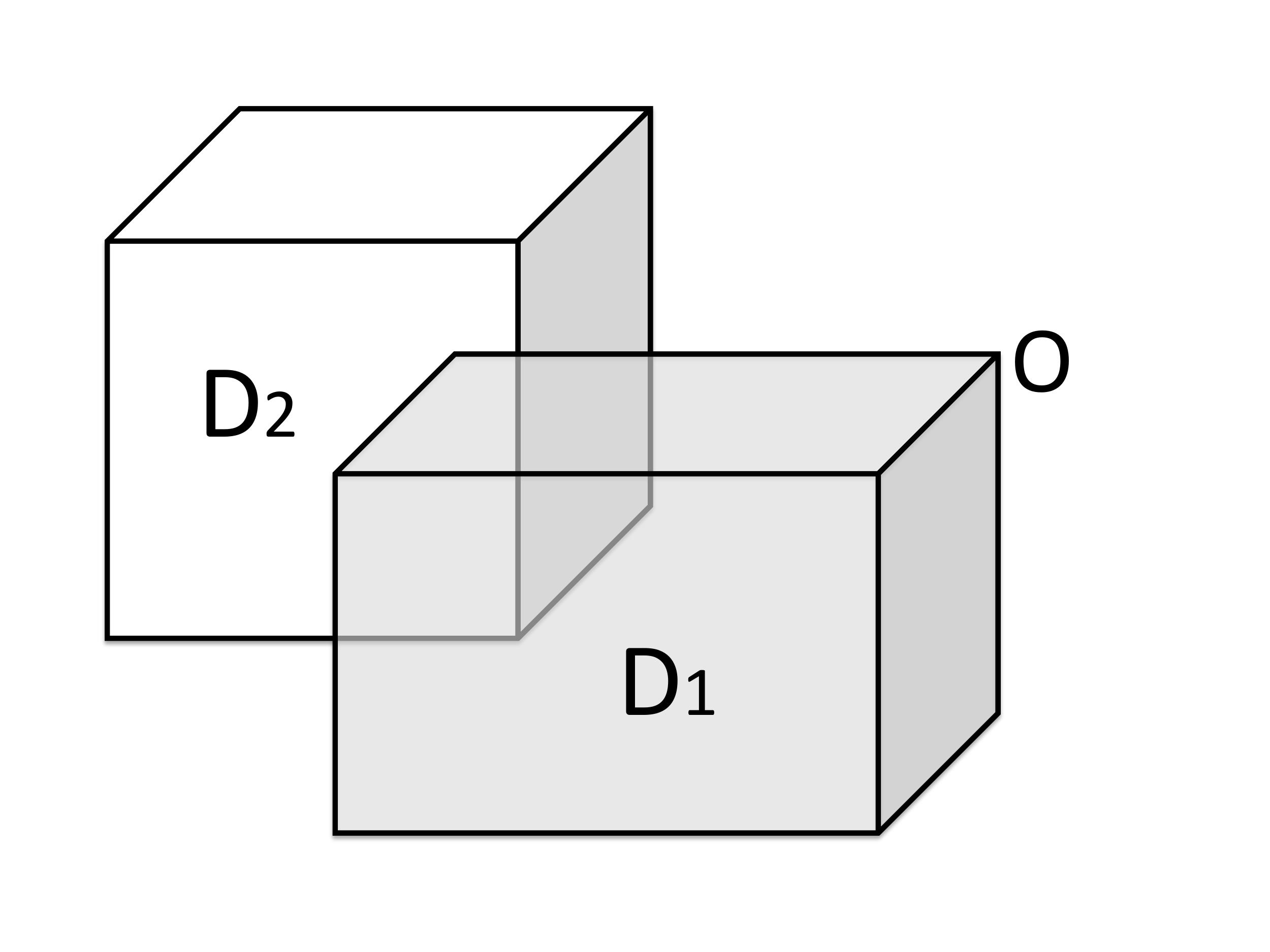}
	\caption{Illustration of two different convex polyhedral scatterers. Left: $D_2$ is a cube and $D_1$ is the interior of  $\overline{D_2\cup G}$, where $G$ denotes the gap domain between $D_1$ and $D_2$. There are four faces of $D_1$ around the vertex $O$, none of them extends to an entire plane in $\mathbb{R}^3\backslash\overline{D}_2$. Right: $D_1$ and $D_2$ are both cubes. There are three faces of $D_1$ around the vertex $O$, only one of them can be extended to an entire plane in $\mathbb{R}^3\backslash\overline{D}_2$.
	}
	\label{f1}
\end{figure}

	\textbf{Proof of Theorem \ref{Main theorem scattering} for electromagnetic Herglotz waves}.
In this case, $E_2$ satisfies the Maxwell's equations in $\Rb^3\backslash\overline{D}_2$. We consider two cases.

Case (i): One of $\widetilde{\Pi}_{j}$ coincides with some hyper-plane $\Pi$ in $\Rb^{3}\backslash{\overline{D_{2}}}$ (see Figure \ref{f1},right).
Since $D_{2}$ is convex,  it must lie completely on one side of the plane $\Pi$. By Corollary \ref{cor3.1}, the electric field $E_{2}$ can be analytically extended to $\Rb^{3}$ as a solution to the Maxwell's equations. This implies that $E_{2}^{sc}$ is an entire radiating solution to the Maxwell's equations. Consequently, we get $E_{2}^{sc}\equiv0$ and thus the total field $E_{2}=E^{in}$ satisfies the impedance boundary condition on $\PD D_{2}$.

Case (ii): None of $\widetilde{\Pi}_{j}$ coincides with an entire hyper-plane in $\Rb^{3}\backslash{\overline{D_{2}}}$ (see Figure \ref{f1}, left). Denote by $\Pi_{j}\supset \widetilde{\Pi}_j$ the hyper-plane in $\Rb^{3}$ containing $\Lambda_{j}$.
We shall prove via reflection principle that $E_2$ satisfies the impedance boundary condition on each $\Pi_j$, which again leads to the relation  $E_{2}=E^{in}$ by repeating the same arguments in case (i).
Without loss of generality we take $j=1$ and consider the plane $\Pi_1\supset\widetilde{\Pi}_1\supset\Lambda_1$. Recall that $\Lambda_1\subset \partial D_1$ is a convex polygonal face and that the total field $E_2$ is analytic near the corner $O$ of $\partial\Lambda_1$. It suffices to prove that $E_2$ is analytic on $\partial \Lambda_1$.
Let $O_1\in\partial\Lambda_1$ be a neighboring corner of $O$, which is also a vertex of $D_2$. By the convexity of $D_2$, there exists at least one face $\Lambda_{j'}$ with $j'\neq 1$ such that the finite line segment $OO_1$ lies completely on one side of the hyper-plane $\Pi_{j'}$ and the projection of $OO_1$ onto $\Pi_{j'}$, which we denote by $L$,  is a subset of $\widetilde{\Pi}_{j'}(\subset\Pi_{j'})$. We refer to Figure \ref{f3} for an illustration of the proof in two dimensions.
Since $D_2$ does not intersect with $\widetilde{\Pi}_{j'}$,  one can always find a symmetric domain $\Omega\subset\Rb^3\backslash \overline{D}_2$ with respect to $\Pi_j$ such that $OO_1\subset \overline{\Omega}$ and $L\subset (\Omega\cap \widetilde{\Pi}_{j'})$.
Recall that $E_2$ fulfills the impedance boundary condition on $\widetilde{\Pi}_{j'}$. Now applying Corollary \ref{cor3.2} with $D=OO_1$ to $E_2$, we find that $E_2$ must be analytic on $\overline{OO}_1$, and in particular, $E_2$ is analytic near $O_1$. Here we have used the fact that the reflection of $\overline{OO_1}$ with respect to $\Pi_{j'}$ lies completely in $\Rb^3\backslash\overline{D}_2$ and $E_2$ is real-analytic in $\Rb^3\backslash\overline{D}_2$.
Analogously, one can prove the analyticity of $E_2$ at another neighboring corner point $O_2$ to $O\in\partial\Lambda_1$ and also the analyticity on the line segment $OO_2\subset \partial \Lambda_1$.  Applying the same arguments to $O_1$ and $O_2$ in place of $O$, we can conclude that $E_2$ is analytic on the closure of $\Lambda_1$. This implies that $E_2$ satisfies the impedance boundary condition on the entire plane $\Pi_1\supset\widetilde{\Pi}_1$ and thus $E_2=E^{in}$ in $\Rb^3$.
	\begin{figure}[htbp]
	\centering
	\includegraphics[width=5.5cm,height=4cm]{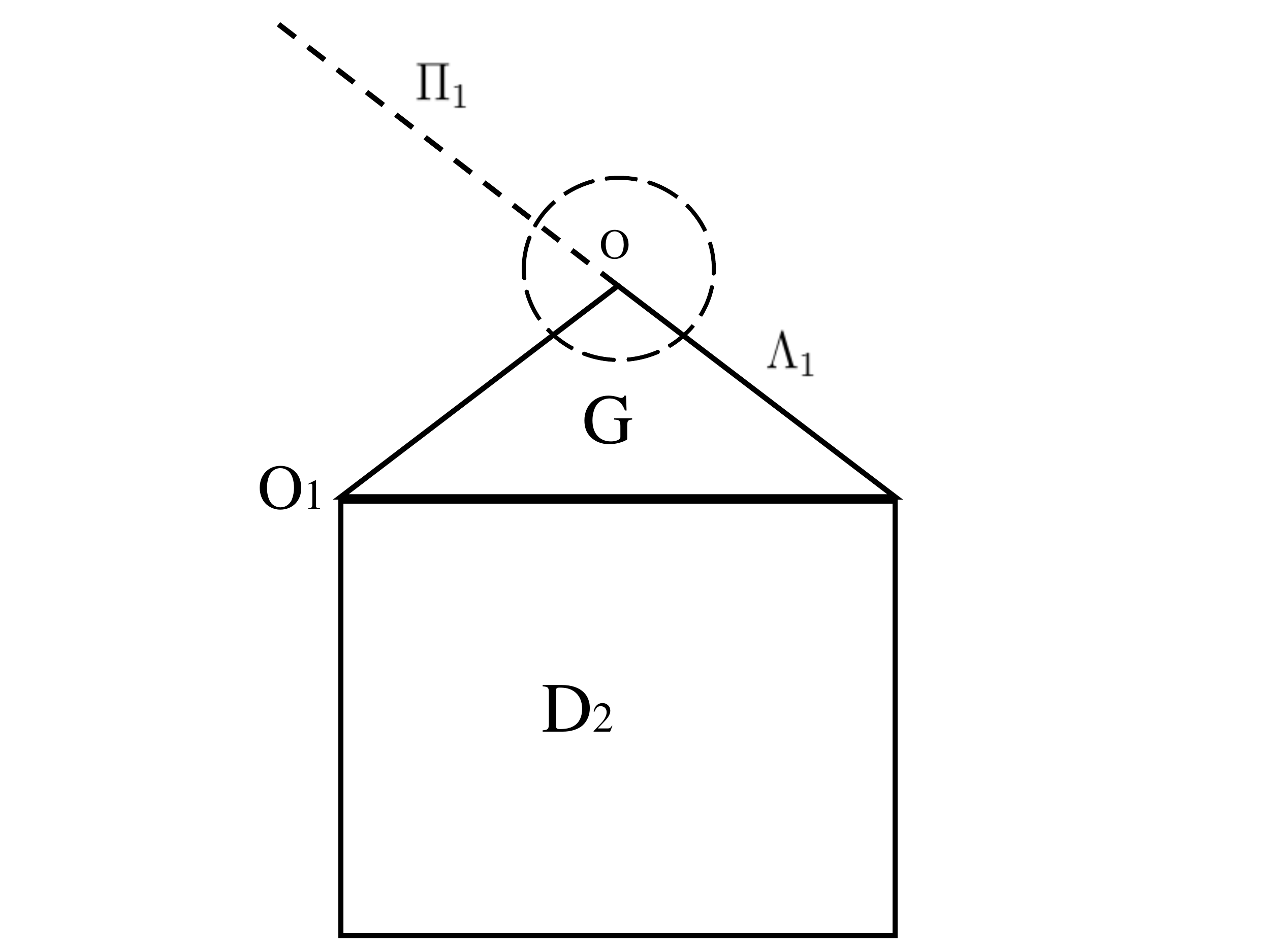}
	\caption{Illustration of two different convex polygonal scatterers: $D_2$ is a square and $D_1$ is the interior of $\overline{D_2\cup G}$, where $G$ denotes the gap domain between $D_1$ and $D_2$. There are two sides of $D_1$ around the corner $O$, both of them cannot be extended to a straight line in $\mathbb{R}^2\backslash\overline{D}_2$. }
	\label{f3}
\end{figure}
To continue the proof, we recall from cases (i) and (ii) that $E^{sc}\equiv 0$ and
the incident field $E^{in}$ satisfies the following boundary value problem in $D_{2}$
\ben\label{Impedance condition for incident field}
&\nabla\times \lb \nabla\times E^{in}\rb -k^{2}E^{in}=0,\ \mbox{in}\ D_{2},\\
&\nu\times\lb \nabla\times E^{in}\rb +i\lambda \nu\times\lb \nu\times E^{in}\rb=0, \ \mbox{on}\ \PD D_{2}.
\enn
Taking the inner product with $\overline{E^{in}}$, integrating over $D_{2}$ and
using the integration by parts, we obtain
	\ben
\int\limits\limits_{D_{2}}\lvert \nabla\times E^{in}\rvert^{2}-k^{2}\lvert E^{in}\rvert^{2}\,\D x+i\lambda \int\limits\limits_{\PD D_{2}}\lvert \nu\times E^{in}\rvert^{2}\D s=0.
\enn
Taking the imaginary parts in the above integral
gives $\nu\times E^{in}=0$ on $\partial D_2$, which together with the impedance boundary condition implies further that $\nu\times (\nabla\times E^{in})=0$ on $\partial D_2$.
	Finally, using the analogue of Holmgren's theorem for Maxwell's equations (see e.g., \cite[Theorem 6.5 ]{Colton_Kress_Book}), we get $E^{in}\equiv 0$, which is a contradiction. Therefore we have $D_{1}=D_{2}$. This proves Theorem \ref{Main theorem scattering} for incident fields given by electromagnetic Herglotz waves. \hfill $\Box$

		\textbf{Proof of Theorem \ref{Main theorem scattering} for electromagnetic point source waves.}
	Suppose that $E^{in}=E^{in}(x,y)$ is an electric point source incited by a magnetic dipole located at $y\in \Rb^3\backslash\overline{D}_j$ ($j=1,2$).
	By the previous proof for electromagnetic Herglotz waves, one can always find $m\geq 3$ entire planes $\Pi_j$ ($j=1,2,\cdots m$) meeting at the vertex $O\in\partial D_1$ such that $E_2$ fulfills the impedance boundary condition on $\Pi_j$, $j=1,2,\cdots,m$. Further, for each $j=1,2,\cdots,m$, $\Pi_j$ does not pass through the source position $y $, and
	the convex polyhedron $D_2$ lies in one side of $\Pi_j$. Repeating the same arguments in case (ii) of the proof for plane waves, one can prove that $E_2$ is analytic on the closure of each face of $D_2$. In particular, $E_2$ fulfills the impedance boundary condition on the entire plane $\Pi_{\Lambda}$ which extends a face $\Lambda$ of $\partial D_2$. By the arbitrariness of $\Lambda$,  we can always find a face $\Lambda'\subset\partial D_2$
	such that the reflection of $y$ with respect to $\Pi_{\Lambda'}$ belongs to $\Rb^3\backslash\overline{D}_2$. By the reflection principle (see Corollary \ref{cor3.2}), $E_2$ must be analytic at $y$, which is a contradiction to the singularity of $E_2$ at the source position.
	\hfill $\Box$
		\subsection{Remarks and Corollaries}\label{sec:remarks}
	Below we present several remarks concerning the  proof of Theorem \ref{Main theorem scattering}.
	\begin{remark}
		(i) Using the reflection principle for the Helmholtz equation (see Theorem \ref{Main theorem reflection} or \cite{Cheng2003}), the idea in the proof of Theorem \ref{Main theorem scattering} can be used to prove unique determination of a convex polyhedral or polygonal scatterer of acoustically impedance-type with a single incoming wave; see Figure \ref{f3} for an illustration of the uniqueness proof in two dimensions. This improves the result of \cite{Cheng2003} where two incident directions were used in 2D.
			(ii) For non-convex polyhedral scatterers, one cannot find a vertex $O$ around which the total field is analytic under the assumption \eqref{Equality of far-field patterns}. Hence, our uniqueness proof to inverse electromagnetic scattering does not apply to non-convex polyhedrons of impedance type. However, this might be possible if one can establish a reflection principe by removing the geometrical assumption of $\Omega$ made in Theorem \ref{Main theorem reflection} (see e.g. \cite{Elschner_Hu_Elastic} in the elastic case).
		\end{remark}
	As a consequence of the proof of Theorem \ref{Main theorem scattering} we have the following corollaries. In particular,
	the \textquoteleft singularity' of $E^{sc}$ at vertices motivates us to design a data-driven scheme (see Section \ref{numeric} below)  to locate
	all vertices of $D$ so that the position and shape of $D$ can be recovered from a single measurement data.
		\begin{corollary}\label{singular}
		Let $D\subset\Rb^{3}$ be a convex polyhedron and let $E=E^{in}+E^{sc}$ be the solution to Equations \eqref{Maxwell equation in exterior}-\eqref{Boundary condition scattering}. Then  $E$ cannot be analytically extended from $\Rb^3\backslash\overline{D}$ to the interior of $D$ across a vertex of $\PD D$, or equivalently, $E$ cannot be analytic on the vertices of $D$.
	\end{corollary}
	
	\begin{corollary}\label{PEC} Let $D$ be a perfectly conduction polyhedron such that $\Rb^3\backslash\overline{D}$ is connected.
		Suppose that $E=E^{in}+E^{sc}$ is a solution to Equations \eqref{Maxwell equation in exterior}-\eqref{Radiation condition} with the boundary condition $\nu\times E=0$ on $\partial D$. If $E^{in}$ is an incident Herglotz wave, we suppose additionally that $k^2$ is not the  eigenvalue of the operator $\nabla\times\nabla\times$ over $D$ with the boundary condition of vanishing tangential components on $\partial D$.
		Then $\partial D$ can be uniquely determined by a single electric far-field pattern $E^\infty$ over all observation directions. Moreover, $E$ cannot be analytically extended from $\Rb^3\backslash\overline{D}$ to the interior of $D$ across a vertex of $\PD D$.
	\end{corollary}
	\begin{proof} Let $E^{in}$ be an incident plane wave with the incident direction $d\in \Sb^2$ and polarization direction $p\in \Sb^2$.
	Suppose that two perfect polyhedral conductors $D_1$ and $D_2$ generate identical electric far-field patterns but $D_1\neq D_2$.
	Combining the path arguments of
	\cite{Liu_Yamamoto_Zou_Reflection principle_sound-soft_Maxwell} and the uniqueness proof in  Theorem \ref{Main theorem scattering} for incoming waves given by electromagnetic Herglotz waves, one can always find a perfectly conducting hyperplane $\Pi\subset\Rb^3$ such that $D_j$ ($j=1$ or $j=2$) lies completely on one side of $\Pi$. In fact, such a plane $\Pi$ can be found by applying the 'point-to-point' reflection principle with the perfectly conducting boundary condition. This implies that the total electric field $E$ can be analytically extended into the whole space, leading to $E^{sc}\equiv 0$ in $\Rb^3$ and thus $\nu\times E^{in}=0$ on $\partial D_j$. Hence, we get $\nu\times p=0$ for any normal direction on $\partial D$, which is impossible.
	
	If $E^{in}$ is an electric Herglotz function, by the assumption of $k^2$ one can also get the vanishing of $E^{in}$.
	In the case that $E^{in}=E^{in}(x,y)$ is an electric point source wave emitting from the source position $y\in \Rb^3\backslash\overline{D}_j$, one can prove that $E_2$ satisfies the Dirichlet boundary condition on the entire plane which extends a face of $D_2$; see the previous section in the impedance case. Using the \textquotedblleft point-to-point" reflection principle together with the path argument,
	this could lead to the analyticity of $E_2$ at $x=y$, which is a contradiction to the singularity of $E_2$ at the source position; see \cite{HuLiu2014} in the acoustic case.
	
	The impossibility of analytical extension across a vertex can by proved analogously.
\end{proof}
	Note that the perfectly conducting polyhedron in Corollary \ref{PEC} is allowed to be non-convex, but cannot contain two-dimensional screens on its closure. The incident wave appearing in Corollaries \ref{singular} and \ref{PEC} can be a plane wave, Herglotz wave function or a point source wave. 
		\subsection{Green's tensor to the Maxwell's equations in a half space with the impedance boundary condition} As another application of the reflection principle, we derive the Green's tensor $G_I(x,y)\in \Cb^{3\times3}$ for the Maxwell's equations in the half space $\Rb_+^3:=\{x:x_3>0\}$ with the impedance boundary condition enforcing on $\Pi:=\{x:x_3=0\}$, that is, for any constant vector $\vec{a}\in \Rb^3$,
	\begin{align}
	\begin{aligned}
	&\nabla\times \lb \nabla\times G_I(x,y)\vec{a}\rb -k^{2}G_I(x,y)\vec{a}=\delta(x-y)\vec{a},\quad\;\;\quad \mbox{in}\ x_3>0,\\
	&\nu\times\lb \nabla\times G_I(x,y)\vec{a}\rb +i\lambda \nu\times\lb \nu\times G_I(x,y)\vec{a}\rb=0, \qquad \mbox{on}\ x_3=0.
	\end{aligned}
	\end{align}
	For this purpose, we need the free-space Green's tensor given by
	\ben
	G(x,y):=\Phi(x,y)\mathbf{I}+\frac{1}{k^2} \nabla_y\nabla_y\Phi(x,y),\quad x\neq y,
	\enn
	where $\mathbf{I}$ is the $3\times 3$ identity matrix and
	$\nabla_y\nabla_y\Phi(x,y)$ is the Hessian matrix for $\Phi$ defined by
	\ben
	\left(\nabla_y\nabla_y\Phi(x,y)\right)_{l,m}=
	\frac{\partial^2\Phi(x,y)}{\partial_{y_l}\partial_{y_m}},\quad 1\leq l,m\leq 3,\quad y=(y_1,y_2,y_3)\in \Rb^3.
	\enn
	Note that here $\Phi(x,y):=\frac{1}{4\pi}\frac{e^{ik\lvert x-y\rvert}}{\lvert x-y\rvert}$ is the fundamental solution to Helmholtz equation in three dimension.
	
	Following the arguments from  \cite[Corollary 2.2]{Elschner_Hu_Elastic}),
		we can prove
	\begin{lemma}
		Denote by $R_{\Pi}$ the reflection with respect to the plane $\Pi$ and by $\mathcal{D}_{x}$ the action with respect to $x$, of the operator $\mathcal{D}:=(\mathcal{D}_1, \mathcal{D}_2,\mathcal{D}_3)$ defined in \eqref{E3 tilde extension Theorem} and \eqref{Extension formula for Ej Theorem}. Then the impedance Green's tensor $G_I(x,y)$ can be represented as
		\ben
		G_{I}(x,y)=G(x,y)+\mathcal{D}_{x}G(R_{\Pi}x,y),\ x\neq y,\quad x,y\in \Rb_+^3.
		\enn
		Here the action of $\mathcal{D}_x$ on the tensor $G$ is understood column-wisely.
	\end{lemma}
	\section{A data-driven imaging scheme}\label{numeric}
The aim of this section is to establish a data-driven inversion scheme for imaging arbitrarily convex-polyhedral scatterers.
Motivated by the one-wave factorization method in inverse elastic scattering \cite{Elschner_Hu_Elastic}, we shall propose a domain-defined indicator functional to characterize an inclusion relationship between a test domain  and our target; see also \cite{HuLi2020} in the acoustic case. Being different from other domain-defined sampling approaches (\cite{KPS, KS, P2003, P2007,PS2010}) arising from inverse scattering, our scheme will be interpreted as a data-driven method, because it  relies on measurement data corresponding to a priori given test domains. In this paper, we shall take for simplicity perfectly conducting balls with different centers and radii as test domains. Similar techniques were used in the Extended Linear Sampling Method \cite{LS} for extracting information of a sound-soft obstacle from a single far-field pattern.

Consider the scattering of an incident plane wave $E^{in}=ik\lb d\times p\rb \times d e^{ikx\cdot d}$ by
a ball $B_{h}(z):=\{x\in\Rb^{3}:\ \lvert x-z\rvert<h\}$ with $h>0, z\in \Rb^3$, where $d\in \Sb^{2}$ is the incident direction and $p\in\Rb^{3}$ is a polarization vector. Then the total field $E=E^{in}+E^{sc}$ satisfies
\begin{align}\label{Sound-soft Maxwell equation}
\begin{aligned}
\begin{cases}
\nabla\times\lb \nabla\times E\rb -k^{2}E=0\quad\,\, \mbox{in}\quad\ \lvert x-z\rvert>h,\\
\lb\nu\times E\rb\times \nu=0\quad\qquad\quad\quad \mbox{on}\quad\ \lvert x-z\rvert=h,\\
\lim_{\lvert x\rvert \rightarrow \infty}\lb E^{sc}\times \widehat{x}+\frac{1}{ik}\nabla\times E^{sc}\rb \lvert x\rvert=0,\ \ \widehat{x}=\frac{x}{r}.
\end{cases}
\end{aligned}
\end{align}
It is well known that \eqref{Sound-soft Maxwell equation} has a series solution $E(\widehat{x};d,p,h,z)$ for a given $E^{in}(x;d,p)$ (\cite{Monk}). For notational convenience we will omit the dependance of solutions on $d$, $p$ and $k$ (all of them are fixed in our arguments) and only indicate the dependance on the center $z\in \Rb^3$ and radius $h>0$ of the ball $B_h(z)$.  
Denote by $E^{\infty}(\widehat{x};h, z)$ the electric far-field pattern of the scattered electric field $E^{sc}$.
We expand $E^{\infty}(\widehat{x};h,z)$ into a series by using vector spherical harmonics. For any orthonormal system $Y^{m}_{n}$, $m=-n,\dots,n$ of spherical harmonics of order $n>0$, the tangential fields defined on the unit sphere
\begin{align*}
U^{m}_{n}\lb \widehat{x}\rb :=\frac{1}{\sqrt{n\lb n+1\rb}}\,\mbox{Grad}\,Y^{m}_{n}(\widehat{x});\ \ V^{m}_{n}(\widehat{x}):=\widehat{x}\times U^{m}_{n}(\widehat{x})
\end{align*}
are called vector spherical harmonics of order $n$. By coordinate translation, it is easy to check that $E^{\infty}$ can be expanded into the convergent series (\cite{CFH03,Monk})
\begin{align}\label{Far-field pattern}
E^{\infty}(\widehat{x};h,z)e^{ikz\cdot \widehat{x}}=4\pi \sum_{n=1}^{\infty}\sum_{m=-n}^{n}\lb u_{n}^{(h)}\, [\overline{U^{m}_{n}}(\widehat{x})\cdot p]\,U^{m}_{n}+v_{n}^{(h)} \,[\overline{V^{m}_{n}}\cdot p]V^{m}_{n}\rb
\end{align}
where \[u^{(h)}_{n}:=\frac{\psi'_{n}(kh)}{(\zeta^{(1)}_{n})'(kh)}\in \mathbb{C} ,\quad v_{n}^{(h)}:=-\frac{\psi(kh)}{\zeta_{n}^{(1)}(kh)}\in\mathbb{C},
\]
with $\psi_{n}(t):=tj_{n}(t)$ and $\zeta_{n}^{(1)}(t):=th_{n}^{(1)}(t).$
Here $j_n$ is the spherical Bessel function of order $n$ and $h_n^{(1)}$ is the spherical Hankel function of first kind of order $n$.
Denote the far-field operator $F^{(z,h)}:T(\Sb^{2})\mapsto T(\Sb^{2})$ by
\begin{align}\label{Far-field operator}
\lb F^{(z,h)}g\rb(\widehat{x}):=\int\limits\limits_{\Sb^{2}}E^{\infty}(\widehat{x},d,g(d),h,z)\D s(d)
\end{align}
where
$T(\Sb^{2}):=\{g\in L^2 (\Sb^2)^3: g(\widehat{x})\cdot\widehat{x}=0$ for all $\widehat{x}\in \Sb^2\}$
denotes the tangential space defined on $\Sb^2$.
The expression \eqref{Far-field pattern} shows that $F^{(z,h)}$ is diagonal in the basis
\[
\widetilde{U}^{(z)}_{m,n}(\widehat{x}):=e^{-ikz\cdot \widehat{x}} U^{m}_{n}(\widehat{x}),\quad
\widetilde{V}^{(z)}_{m,n}(\widehat{x}):=e^{-ikz\cdot \widehat{x}} V^{m}_{n} (\widehat{x}).
\]
It can be  verified that $\lb 4\pi u^{(h)}_{n},4\pi v^{(h)}_{n}\rb$ and $\lb \widetilde{U}^{(z)}_{m,n},\widetilde{V}^{(z)}_{m,n}\rb$ are eigenvalues and the associated eigenvectors of $F^{(z,h)}$. Note that the eigenvalues depend on the radius $h$ only and the eigenfunctions depend on the location $z$ only.
We refer to \cite{CFH03} for detailed analysis when the ball is located at the origin. The general case can be easily justified via coordinate translation.

To proceed, we suppose that $w^\infty\in T(\Sb^2)$ is the electric field pattern of some radiating electric field $w^{sc}$ in $|x|>b$ for some $b>0$ sufficiently large.
Introduce the function
\begin{align}\label{Indicator function}
I_{w^\infty}(z,h):=\frac{1}{4\pi}\sum_{n=0}^{\infty}\sum_{m=-n}^{n}\lb \frac{\lvert \langle w^{\infty},\widetilde{U}^{(z)}_{m,n}\rangle\rvert^{2}}{\lvert u^{(h)}_{n}\rvert}+\frac{\lvert \langle w^{\infty},V^{(z)}_{m,n}\rangle\rvert^{2}}{\lvert v^{(h)}_{n}\rvert}\rb,
\end{align}
where $z\in \Rb^3$ and $h>0$ will be referred to as sampling variables in this paper. Equation \eqref{Indicator function} can be regarded as a functional defined on the test domain $B_h(z)$. If the above series is convergent, we shall prove below that the radiated electric field $w^{sc}$ can be analytically extended at least to the exterior of the test domain $B_h(z)$. For simplicity we still denote by $w^{sc}$ the extended solution.

\begin{lemma}\label{Lemma in numerics} Suppose that $k^2$ is not the Dirichlet eigenvalue (that is, the tangential component {of the electric field} vanishes) of the operator $\rm{curl}\rm{curl}$ over the ball $B_h(z)$.
	We have $I_{w^\infty}(z,h)<\infty$ if and only if $w^{\infty}$ is the far-field pattern of the radiating field $w^{sc}$ which satisfies
	\begin{equation}\label{w}
	\nabla\times\lb\nabla\times w^{sc}\rb -k^{2}w^{sc}=0\quad \mbox{in}\ \lvert x-z\rvert>h,\quad
	\nu\times w^{sc}\times \nu\in H^{-1/2}_{\rm{curl}}(\PD B_{h}(z)).
	\end{equation}
	{	Here  $H^{-1/2}_{\rm{curl}}(\PD D)$ denotes the trace space of $$H\lb \rm{curl},D\rb =\{\phi\in (L^{2}(D))^3:\ \nabla\times \phi\in (L^{2}(D))^3\}$$ of a bounded Lipschitz domain $D\subset \Rb^3$,  given by
		\[
		H^{-1/2}_{\rm{curl}}(\PD D):=\left\{u\in \lb H^{-1/2}(\PD D)\rb^{3}:\ \nu\cdot u=0,  \nabla\times u\in \lb H^{-1/2}(\PD D)\rb^{3} \mbox{on}\; \PD D \right\}.
		\]
	}
	\begin{proof} Without loss of generality we may assume that $B_h(z)$ is located at the origin, so that $\widetilde{U}^{(z)}_{m,n}=U_n^m$ and 	$\widetilde{V}^{(z)}_{m,n}=V_n^m$. Since the assumption on the wavenumber $k$ ensures that (see \cite[Chapter 5]{Nedelec} for related discussions)
		\[
		j_n(t)\neq0\quad \mbox{and}\quad j_n(t)+tj'_n(t)\neq 0 \qquad \mbox{for}\quad t=kh, \;n=1,2\cdots,
		\]
		we have $|u_n^{(h)}|\neq 0$ and $|v_n^{(h)}|\neq 0$ for all $n$.
		By \cite[Equation $6.73$]{Colton_Kress_Book}) it follows that $w^{sc}$ can be expressed as
		\[w^{sc}(x)=\sum_{n=1}^{\infty}\frac{1}{n(n+1)}\sum_{m=-n}^{n} \Big[ a^{m}_{n}q^{m}_{n}(x)+b^{m}_{n}\nabla\times q^{m}_{n}(x)\Big]\quad\mbox{in}\quad |x|>b,\]
		with the coefficients $a_n^m, b_n^m\in \mathbb{C}$ and
		$q^{m}_{n}(x):=\nabla\times \{xh^{(1)}_{n}(k\lvert x\rvert)Y^{m}_{n}(\widehat{x})\}$. Correspondingly, the far-field pattern $w^{\infty}$ is given by (see Equation $(6.74)$ on page $219$ in \cite{Colton_Kress_Book}):
		\[w^{\infty}(\widehat{x})=\frac{i}{k}\sum_{n=1}^{\infty}\frac{1}{i^{n+1}}\sum_{m=-n}^{n}\lb ikb_n^mU^{m}_{n}(\widehat{x})-a^{m}_{n}V^{m}_{n}(\widehat{x})\rb.\]
		Inserting the above expression into \eqref{Indicator function}, we get
		\begin{align}\label{I}
		I_{w^{\infty}}(o,h)=\frac{1}{4\pi}\sum_{n=1}^{\infty}\sum_{m=-n}^{n}\lb \frac{\lvert b^{m}_{n}\rvert^{2}}{\lvert u^{(h)}_{n}\rvert}+\frac{\lvert a^{m}_{n}\rvert^{2}}{\lvert v^{(h)}_{n}\rvert}\rb,\quad o=(0,0,0).
		\end{align}
		To analyze the convergence of the above series, we need the asymptotic behavior of $u_n^{(h)}$ and $v_n^{(h)}$ as $n\rightarrow+\infty$.	Using the asymptotics of special functions for large orders, 	
		it is easy to observe that
		\ben
		&&\frac{1}{\lvert u^{(h)}_{n}\rvert}=\left\lvert \frac{\lb\zeta_{n}^{(1)}\rb'}{\psi_{n}'}\right\rvert=\left\lvert \frac{\lb th_{n}^{(1)}(t)\rb'}{\lb tj_{n}(t)\rb'}\Bigg|_{t=kh}\right\rvert\sim \left\lvert \frac{\lb h^{(1)}_{n}\rb'(kh)}{j_{n}'(kh)}\right\rvert\sim \frac{C_{1}}{\lvert n\rvert}\left\lvert \lb h^{(1)}_{n}\rb'(kh)\right\rvert^{2},\\
		&&\frac{1}{\lvert v^{(h)}_{n}\rvert}=\left\lvert \frac{\zeta^{(1)}_{n}(kh)}{\psi_{n}(kh)}\right\rvert=\left\lvert \frac{h^{(1)}_{n}(kh)}{j_{n}(kh)}\right\rvert \sim C_{2}\,\lvert n\rvert\, \lvert h^{(1)}_{n}(kh)\rvert^{2},\enn
		as $n\rightarrow\infty$, where $C_{1}, C_{2}\in \mathbb{C}$ are fixed constants. Thus it follows from \eqref{I} that
		\begin{align}\label{Approximation for Iws}
		\begin{aligned}
		I_{w^{\infty}}(o,h) \sim \sum_{n=1}^{\infty}\sum_{m=-n}^{n}\lb \frac{C_{1}\lvert b_{n}^{m}\rvert^{2}\lvert h^{(1)}_{n}{'}(kh)\rvert^{2}}{\lvert n\rvert}+C_{2}\lvert a^{m}_{n}\rvert^{2}\lvert n\rvert \lvert h^{(1)}_{n}(kh)\rvert^{2}\rb
		\end{aligned}
		\end{align}
		On the other hand, it is seen from the expression of $w^{sc}$ that on $|x|=h$, 	\ben
		\lb \widehat{x}\times w^{sc}\times \widehat{x}\rb=\sum_{n=1}^{\infty}\sum_{m=-n}^{n}\left\{\frac{b^{m}_{n}}{h}\lb th^{(1)}_{n}(t)\rb' \Big|_{t=kh}U^{m}_{n}(\widehat{x})-a^{m}_{n}h^{(1)}_{n}(kh)V^{m}_{n}(\widehat{x})\right\}.
		\enn
		By definition of the $H^{-1/2}_{\rm{curl}}(\PD B_{h}(z))$ norm (see e.g. \cite[Chapter 5]{Nedelec} and \cite[Chapter 9.3.3]{Monk}) we obtain
		\begin{align}\nonumber
		&\lVert \widehat{x}\times w^{sc}\times \widehat{x}\rVert_{H^{-1/2}_{\rm{curl}}(\PD B_{h}(z))}\\ \nonumber
		&=\sum_{n=1}^{\infty}\sum_{m=-n}^{n} \lb\frac{1}{\sqrt{n\lb n+1\rb}}\frac{\lvert b^{m}_{n}\rvert^{2}}{\lvert h\rvert^{2}}\left[\lb th^{(1)}_{n}(t)\rb'\Big|_{t=kh}\right]^2+\sqrt{n\lb n+1\rb}\lvert a^{m}_{n}\rvert^{2}\lvert h^{(1)}_{n}(kh)\rvert^{2}\rb\\ \label{Approximation for norm f}
		&\ \ \sim \sum_{n=1}^{\infty}\sum_{m=-n}^{n} \lb\frac{1}{n}\lvert b^{m}_{n}\rvert^{2}\lvert h^{(1)}_{n}{'}(kh)\rvert^{2}+n\lvert a^{m}_{n}\rvert^{2}\lvert h^{(1)}_{n}(kh)\rvert^{2}\rb.
		\end{align}
		Obviously, \eqref{Approximation for Iws} and \eqref{Approximation for norm f} have the same convergence. In the same manner, one can prove that \eqref{Approximation for Iws} has the same convergence with
		$||\nu\times w^{sc}||_{H_{\rm div}^{-1/2}(\partial B_h(z))}$ where
		\[
		H^{-1/2}_{\rm{div}}(\PD D):=\{u\in \lb H^{-1/2}(\PD D)\rb^{3}:\ \nu\cdot u=0\ \mbox{on}\; \PD D\; \mbox{and}\ \lb \rm{Div}\, u\rb\in H^{-1/2}(\PD D)\}.
		\]
		Using the relation
		\begin{eqnarray*}
			&&||\nu\times w^{sc}||_{L^{2}(\partial B_h(z))}+
			||\nu\times w^{sc}\times\nu||_{L^{2}(\partial B_h(z))}\\
			&\leq& C\,\lb ||\nu\times w^{sc}||_{H_{\rm div}^{-1/2}(\partial B_h(z))}+
			||\nu\times w^{sc}\times \nu||_{H_{\rm curl}^{-1/2}(\partial B_h(z))}\rb,
		\end{eqnarray*}
		we conclude that the tangential components of $w^{sc}$ on $\partial B_h(z)$ are convergent in the $L^2$-sense, if $I_{w^{\infty}}(o,h)<\infty$.  This together with
		\cite[Theorem 6.27]{Colton_Kress_Book} implies that $w^{sc}$ is a solution to the Maxwell's equations in $|x|>h$.
		The proof of Lemma \ref{Lemma in numerics} is thus complete.
	\end{proof}
\end{lemma}
Combining Lemma \ref{Lemma in numerics} and Corollary \ref{singular}, we may characterize an inclusion relation between $D$ and $B_h(z)$ through the measurement data $E^\infty$ of our target and
the spectra of the far-field operator $F^{(z,h)}$ corresponding to the test ball.
\begin{theorem}\label{th4.2}
	Let $E^\infty$ be the electric far-field pattern of a convex-polyhedral scatterer $D$ with a constant impedance coefficient. Suppose that $k^2$ is not the Dirichlet eigenvalue  of the operator $\rm{curl}\rm{curl}$ over the ball $B_h(z)$.
	It holds that
	\[ I_{E^\infty}(z, h)<\infty\qquad \mbox{if and only if}\qquad D\subset \overline{B_z(h)}.
	\]
	Hence we have \[
	D=\bigcap^{(z,h)}_{I_{E^\infty}(z, h)<\infty} B_h(z).\]
\end{theorem}
\begin{proof}
	If $D\subset \overline{B_z(h)}$, the scattered electric field $E^{sc}$ is well defined in $\{x: |x-z|>h\}$ which lies in the exterior of $D$. Hence, $E^{sc}$ satisfies \eqref{w} and by Lemma \ref{Lemma in numerics} it holds that $I_{E^\infty}(z, h)<\infty$. On the other hand, suppose that $I_{E^\infty}(z, h)<\infty$
	but the relation $D\subset \overline{B_z(h)}$ does not hold. Since $D$ is a convex polyhedron, there must exist at least one vertex $O$ of $\partial D$ such that $|O-z|>h$. Again using Lemma \ref{Lemma in numerics}, we conclude that $E^{sc}$ can be extended from $\Rb^3\backslash \overline{D}$ to the exterior of $B_h(z)$. This implies that $E^{sc}$ is analytic at $O$, which contradicts Corollary \ref{singular}.
\end{proof}

By Theorem \ref{th4.2}, the function $h\rightarrow I_{E^\infty}(z, h)$ for fixed $|z|=R$ will blow up when $h\geq \max_{y\in\partial D}|y-z|$, indicating a rough location of $D$ with respect to $z\in \Rb^3$.
In Table \ref{table} we describe an inversion procedure for imaging
an arbitrary convex-polyhedron $D$ by taking both $z\in \partial B_R$ and $h$ as sampling variables.
The mesh for discretizing $h\in(0,2R)$ should be finer than the mesh for $z\in \partial B_R$. To avoid the {assumption} that $k^2$ is not a Dirichlet eigenvalue of ${\rm curl\,curl}$ over $B_{h}(z)$, one may use coated balls by a thin dielectric layer (which can be modeled by the impedance boundary condition)  as test domains in place of our choice of perfectly conducting balls. We refer to \cite[Section 3.1]{CFH03} for a description of the spectra of the far-field operator corresponding to such coated balls centered at the origin. If the impedance coefficient is a positive constant, one can prove that $k^2 $ cannot be an impedance eigenvalue of {\rm curl\,curl} over any boundary Lipschitz domain. It should be remarked that the test domains can also be taken as penetrable balls under the assumption that $k^2$ is not an interior transmission eigenvalue.  Both Theorem \ref{th4.2} and Lemma  \ref{Lemma in numerics} can be carried over to these test domains. Finally, it is worth mentioning that a regularization scheme should be employed to truncate the series \eqref{I}, because the eigenvalues $u_n^{(h)}$ and $v_n^{(h)}$ decay very fast and the calculation of the inner product between $E^\infty$ and the eigenfunctions $\lb \widetilde{U}_{n,m}^{(z)}, \widetilde{V}_{n,m}^{(z)} \rb$ is usually polluted by data noise and numerical errors. We refer to \cite{MH} for numerical examples in inverse acoustic scattering. Numerical tests for Maxwell's equations will be reported in our forthcoming publications.

\begin{table}\label{table}\caption{Data-driven scheme for imaging convex polyhedral scatterers}
	\begin{tabularx}{\textwidth }{>{\bfseries}lX}
		\toprule
		Step 1 & Collect the measurement data $E^{\infty}(\widehat{x})$ for all $\widehat{x} \in \mathbb{S}^2$ and suppose that $D\subset B_R:=\{x: |x|<R\}$ for some large $R>0$. \\ \midrule
		Step 2 & Choose sampling variables $z_{j} \in \{x:\ |x|=R\}$ and  $h_i\in (0, 2R)$ to get the spectra of the far-field operator $F^{(z,h)}$ corresponding to  testing balls $B_{h_j}(z_j)\subset B_R$.\\ \midrule
		Step 3 & Calculate the domain-defined indicator function $I_{E^\infty}(z_j, h_i)$ by \eqref{Indicator function} with $w^\infty=E^\infty$. In particular, it follows from Theorem \ref{th4.2} that
		\[h_i< \max_{y\in\partial D} |z_j-y|\longrightarrow I_{E^\infty}(z_j, h_i)=\infty,\]
		\[
		h_i\geq\max_{y\in\partial D} |z_j-y|\longrightarrow I_{E^\infty}(z_j, h_i)<\infty.\]
		\\ \midrule
		Step 4 &  Image $D$ as the intersection of all test balls $B_{h_i}(z_j)$ such that $I_{E^\infty}(z_j, h_i)<\infty$. \\
		\bottomrule
	\end{tabularx}
\end{table}
\begin{remark}
	In \cite{PS2010}, the No Response Test was discussed for reconstructing convex perfectly conducting polyhedrons with two or a few incident electromagnetic plane waves. In comparison with \cite{PS2010}, our inversion scheme uses only a single incoming wave within a more general class of plane waves, Herglotz wave functions and point source waves. Although both of them belong to the class of domain-defined sampling methods, the computational criterion explored here (see (\ref{Indicator function}) and Theorem \ref{th4.2}) involves simple inner product calculations and new sampling schemes due to the special choice of testing balls.
\end{remark}
	   \section*{Acknowledgments}
	   {The authors would like to thank the three anonymous referees for their comments and suggestions which help improve the original manuscript.}

   \end{document}